\documentclass{amsart}
\usepackage{amssymb,amsmath,amsthm,amsfonts}
\usepackage {latexsym}
\usepackage{bbm}

\allowdisplaybreaks
\newcommand{\nc}{\newcommand}
\nc{\les}{\lesssim}
\nc{\nit}{\noindent}
\nc{\nn}{\nonumber}
\nc{\D}{\partial}
\nc{\diff}[2]{\frac{d #1}{d #2}}
\nc{\diffn}[3]{\frac{d^{#3} #1}{d {#2}^{#3}}}
\nc{\pdiff}[2]{\frac{\partial #1}{\partial #2}}
\nc{\pdiffn}[3]{\frac{\partial^{#3} #1}{\partial{#2}^{#3}}}
\nc{\abs}[1] {\lvert #1 \rvert}
\nc{\cAc}{{\cal A}_c}
\nc{\cE}{{\cal E}}
\nc{\cF}{{\cal F}}
\nc{\cP}{{\cal P}}
\nc{\cV}{{\cal V}}
\nc{\cQ}{{\cal Q}}
\nc{\cGin}{{\cal G}_{\rm in}}
\nc{\cGout}{{\cal G}_{\rm out}}
\nc{\cO}{{\cal O}}
\nc{\Lav}{{\cal L}_{\rm av}}
\nc{\cL}{{\cal L}}
\nc{\cB}{{\cal B}}
\nc{\cZ}{{\cal Z}}
\nc{\cR}{{\cal R}}
\nc{\cT}{{\cal T}}
\nc{\cY}{{\cal Y}}
\nc{\cX}{{\cal X}}
\nc{\cXT}{{{\cal X}(T)}}
\nc{\cBT}{{{\cal B}(T)}}
\nc{\vD}{{\vec \mathcal{D}}}
\nc{\efield}{\mathcal{E}}
\nc{\vE}{{\vec \efield}}
\nc{\vB}{{\vec \mathcal{B}}}
\nc{\vH}{{\vec \mathcal{H}}}
\nc{\ty}{{\tilde y}}
\nc{\tu}{{\tilde u}}
\nc{\tV}{{\tilde V}}
\nc{\Pc}{{\bf P_c}}
\nc{\bx}{{\bf x}}
\nc{\bX}{{\bf X}}
\nc{\bXYZ}{{\bf XYZ}}
\nc{\bY}{{\bf Y}}
\nc{\bF}{{\bf F}}
\nc{\bS}{{\bf S}}
\nc{\dV}{{\delta V}}
\nc{\dE}{{\delta E}}
\nc{\TT}{{\Theta}}
\nc{\dPsi}{{\delta\Psi}}
\nc{\order}{{\cal O}}
\nc{\Rout}{R_{\rm out}}
\nc{\eplus}{e_+}
\nc{\eminus}{e_-}
\nc{\epm}{e_\pm}
\nc{\eps}{\varepsilon}
\nc{\vnabla}{{\vec\nabla}}
\nc{\G}{\Gamma}
\nc{\w}{\omega}
\nc{\mh}{h}
\nc{\mg}{g}
\nc{\vphi}{\varphi}
\nc{\tlambda}{\tilde\lambda}
\nc{\be}{\begin{equation}}
\nc{\ee}{\end{equation}}
\nc{\ba}{\begin{eqnarray}}
\nc{\ea}{\end{eqnarray}}

\nc{\g}{\gamma}
\nc{\ol}{\overline}

\newtheorem{theorem}{Theorem}[section]
\newtheorem{lemma}[theorem]{Lemma}
\newtheorem{prop}[theorem]{Proposition}
\newtheorem{corollary}[theorem]{Corollary}
\newtheorem{defin}[theorem]{Definition}
\newtheorem{rmk}[theorem]{Remark}

\nc{\pT}{\partial_T}
\nc{\pz}{\partial_z}
\nc{\pt}{\partial_t}
\nc{\la}{\langle}
\nc{\ra}{\rangle}
\nc{\infint}{\int_{-\infty}^{\infty}}
\nc{\halfwidth}{6.5cm}
\nc{\figwidth}{10cm}
\newcommand{\f}{\frac}

\nc{\nlayers}{L} \nc{\nsectors}{M}
\nc{\indicator}{\mathbf{1}}
\nc{\Rhole}{R_{\rm hole}}
\nc{\Rring}{R_{\rm ring}}
\nc{\neff}{n_{\rm eff}}
\nc{\Frem}{F_{\rm rem}}
\nc{\R}{\mathbb R}
\nc{\Z}{\mathbb Z}
\nc{\DD}{\Delta}
\nc{\cD}{\mathcal D}
\nc{\lnorm}{\left\|}
\nc{\rnorm}{\right\|}
\nc{\rnormp}{\right\|_{\ell^{p,\eps}}}
\nc{\rar}{\rightarrow}
\date{\today}
\begin{document}

\begin{abstract}

We study the wave equation with potential
$u_{tt}-\Delta u+Vu=0$ in two spatial dimensions,
with $V$  a real-valued, decaying potential.
With $H=-\Delta+V$, 
we study a variety of mapping estimates of the
solution operators, $\cos(t\sqrt{H})$ and
$\frac{\sin(t\sqrt{H})}{\sqrt{H}}$ under the assumption that
zero is a regular point of the spectrum of $H$.  
We prove a dispersive
estimate with a time decay rate of $|t|^{-\frac{1}{2}}$,
a polynomially weighted dispersive estimate which
attains a faster decay rate of $|t|^{-1}(\log |t|)^{-2}$
for $|t|>2$.  Finally, we prove dispersive estimates
if zero is not a regular point of the spectrum of
$H$.

\end{abstract}

\title{Time decay estimates for the wave equation with 
potential in dimension two}

\author{William~R. Green}

\address{Department of Mathematics\\
Rose-Hulman Institute of Technology \\
Terre Haute, IN 47803, U.S.A.}
\email{green@rose-hulman.edu}
\thanks{The author gratefully 
acknowledges the support of an 
AMS-Simons Travel Grant.}

\maketitle
\section{Introduction}
In this paper we study the linear wave equation 
with a real-valued decaying 
potential in two spatial dimensions,
\begin{align}
	u_{tt}-\Delta u+Vu=0, \qquad 
	u(0)=f, \quad \partial_t u(0)=g,
	\quad(x,t)\in\R^2 \times \R.\label{wave eqn}
\end{align}
Formally, with $H=-\Delta +V$, 
we represent the solution of this equation
by
\begin{align}\label{wave soln}
	u(x,t)=\cos(t\sqrt{H})f(x)+\frac{\sin(t\sqrt{H})}
	{\sqrt{H}}g(x).
\end{align}
This formulation is valid if, for example, $f\in L^2$ and 
$g\in \dot{H}^{-1}$.  In the free case, when $V=0$,
it is known that if the inital data $f$ and $g$ lie 
in sufficiently regular Sobolev spaces we also have the
dispersive estimates
\begin{align}
	&\|\cos(t\sqrt{-\Delta})f\|_{\infty}\les |t|^{-1/2}
	\|f\|_{W^{k+1,1}},\\
	&\bigg\|\frac{\sin(t\sqrt{-\Delta})}{\sqrt{-\Delta}}g
	\bigg\|_{\infty}\les |t|^{-1/2}
	\|g\|_{W^{k,1}},
\end{align}
for $k>\frac{1}{2}$ in dimension two, see 
\cite{MSW,Beals, BS}.  In general, one must use Hardy or
Besov spaces or BMO to obtain the sharp $k=\frac{n-1}{2}$ 
smoothness bound in even dimensions.  One can
attain the bound for Sobolev spaces in dimension three
where one can use the divergence 
theorem, see \cite{Str}.  More generally, 
in odd spatial dimensions one
can take advantage of the strong Huygens principle.

In this paper, we extend these dispersive bounds to the 
perturbed equation, \eqref{wave eqn}, when the initial data
$(f,g)$ is sufficiently differentiable with relatively
weak assumptions on the potential $V$.  These, along with
the $H^s$ bounds yield a class of $L^p$ dispersive
bounds, see Theorem~1.1 of \cite{Beals}.

Such dispersive estimates were first studied by 
Beals and Strauss in dimensions $n\geq 3$ in \cite{BS}.
In two dimensions there is not much work on the
$W^{k,1}\to L^\infty$ dispersive estimates or 
`regularized'
$L^1\to L ^\infty$ type estimates, where negative
powers of $H$ are employed. 
High frequency estimates of this type were studied 
by Moulin
in \cite{Mou}.
In 
\cite{Kop}, Kopylova studied two dimensional local
estimates, based on polynomially weighted $H^s$ spaces
when zero energy is regular, in which a decay rate
of $t^{-1}(\log t)^{-2}$ is attained for large $t$.  
In the $H^s$ setting, one can take advantage of conservation
of energy, which is unavailable for the global
dispersive bounds.  This
matches the decay rate obtained by Murata for the 
two dimensional Schr\"odinger equation in \cite{Mur}.
Dispersive bounds for the wave equation in three 
dimensions have been studied, see \cite{KS,BeGol} for
example.
More is known about Strichartz estimates for the wave
equation, particularly in dimensions $n\geq 3$ see
for example \cite{GinV,Burq1,Burq2,GV,DP,BeGol}.

When we perturb the free wave equation with a potential,
one needs to project away from the eigenspaces.  If
we assume that $|V(x)|\les \la x\ra^{-\beta}$
for some $\beta>2$, it is
well known, see \cite{RS1}, that the spectrum of
$H$ is made of finitely many non-negative eigenvalues
and the absolutely continuous spectrum $[0,\infty)$. 
The negative eigenvalues cause exponential growth
of the solutions due to the
action of the operators $\cos(t\sqrt{E})$ and 
$E^{-\f12}\sin(t\sqrt{E})$ with $\sqrt{E}\in i\R$.
Accordingly, we first project away from the eigenvalues
with $P_{ac}$, the projection onto the absolutely
continuous spectrum.  

Throughout the paper we use the
notation $a-:=a-\epsilon$ and $a+:=a+\epsilon$ for
some small, but fixed, $\epsilon>0$.
Our first main result is the
dispersive bound

\begin{theorem}\label{thm:main}

	If $|V(x)|\les \la x\ra^{-\beta}$ for some
	$\beta>3$ and if zero is a regular point of the 
	spectrum of $H=-\Delta+V$, then
	\begin{align*}
		&\|\cos(t\sqrt{H}) \la H\ra^{-\f34-}P_{ac}(H)f
		\|_{L^1\to L^\infty}\les |t|^{-\f12}, \\
		&\bigg\|\frac{\sin(t\sqrt{H})}{\sqrt{H}}
		\la H\ra^{-\f14-} P_{ac}(H)g
		\bigg\|_{L^1\to L^\infty}\les |t|^{-\f12}.
	\end{align*}

\end{theorem}

In addition, we prove a weighted version of this
Theorem that attains a faster time decay rate. 

\begin{theorem}\label{thm:wtd}

	If $|V(x)|\les \la x\ra^{-\beta}$ for some
	$\beta>3$ and if zero is a regular point of the 
	spectrum of $H=-\Delta+V$, then for $t>2$
	\begin{align*}
		&\|\cos(t\sqrt{H}) \la H\ra^{-\f34-}P_{ac}(H)f
		\|_{L^{1,\f12+}\to L^{\infty, -\f12-}}\les t^{-1}
		(\log t)^{-2}, \\
		&\bigg\|\frac{\sin(t\sqrt{H})}{\sqrt{H}}
		\la H\ra^{-\f14-} P_{ac}(H)g
		\bigg\|_{L^{1,\f12+}\to L^{\infty, -\f12-}}
		\les t^{-1} (\log t)^{-2}.
	\end{align*}

\end{theorem}
Here $L^{1,\f12+}:=\{f:\R^2\to \mathbb C \, :\, \|f\la \cdot\ra^{\f12+}\|_1<\infty\}$ and $L^{\infty, -\f12-}$ is
defined similarly.  We note that the bounds in 
Theorems~\ref{thm:main} and \ref{thm:wtd} can be replaced
by bounds withought negative powers of $\la H \ra$ which
map $W^{2,1}\to L^\infty$ for the cosine 
operator and $W^{1,1}\to L^\infty$ for the sine 
operator (and respective weighted Sobolev spaces for
Theorem~\ref{thm:wtd}), see Remark~\ref{rmk:high}
below.  We note that the regularizing powers of
$\la H\ra^{-\alpha}$ for $\alpha>0$ reflect the loss of
derivatives of the initial data, and
are only needed
for high energy.  In the low energy estimates, we 
prove $L^1\to L^\infty$ bounds. 

A series of papers have considered such estimates involving `regularizing' powers of $H$, mostly 
with a high-energy cut-off and using
semi-classical techniques.  See \cite{CCV,Mou} for two dimensional results.  In higher
dimensions see, for example, \cite{Vod1,CV}.

In the case that zero is not regular, that is there are
non-trivial bounded solutions to $H\psi=0$, one can study the
boundedness of sine and cosine operators.  The 
distributional solutions
of this equation have been  
characterized, see
\cite{JN,EG}, in terms of the $L^p$ spaces.  If 
$\psi\in L^{\infty}$ and $\psi\notin L^p$ for any $p<\infty$
we call $\psi$ an `s-wave' resonance.  If $\psi\in L^p$
for all $p>2$ but not in $L^2$, we say $\psi$ is a
`p-wave' resonance.  Whereas, if $\psi \in L^2$ there is an
eigenvalue at zero.  It is known that when zero is not
regular, in general, time decay is lost in the dispersive
estimates.  See, for example, \cite{JenKat,ES} for three-dimensional and  \cite{EG} for two-dimensional
Schr\"odinger operators.  In the two-dimensional Schr\"odinger equation, it is shown that the `s-wave' resonance is sufficiently mild as to not destroy the dispersive time decay rate.  We prove a similar characterization for the wave equation.

\begin{theorem}\label{thm:nonreg}

	Assume that $|V(x)|\lesssim \langle x\rangle^{-\beta}$.  If there is only an s-wave resonance
	at zero energy, then for $\beta>4$, for $|t|>1$ we have
	\begin{align*}
		&\|\cos(t\sqrt{H}) \la H\ra^{-\f34-}P_{ac}(H)f
		\|_{L^1\to L^\infty}\les |t|^{-\f12}, \\
		&\bigg\|\frac{\sin(t\sqrt{H})}{\sqrt{H}}
		\la H\ra^{-\f14-} P_{ac}(H)g
		\bigg\|_{L^1\to L^\infty}\les |t|^{-\f12}.
	\end{align*}
	If there is a p-wave resonance or eigenvalue at zero, then for $\beta>6$, there are time-dependent operators
	$F_t$ and $G_t$ such that for $|t|>1$,
	\begin{align*}
		&\|\cos(t\sqrt{H})[\la H\ra^{-\f34-}
		P_{ac}(H)-F_t]f
		\|_{L^1\to L^\infty}\les |t|^{-\f12}, \\
		&\bigg\|\bigg[
		\frac{\sin(t\sqrt{H})}{\sqrt{H}}\la H\ra^{-\f14-} 
		P_{ac}(H)-G_t\bigg]g
		\bigg\|_{L^1\to L^\infty}\les |t|^{-\f12}.
	\end{align*}
	with
	\begin{align*}
	\sup_t \| F_t \|_{L^{1}\to L^\infty}\lesssim 1,
	\qquad \textrm{ and } \qquad  \| G_t \|_{L^{1}\to L^\infty}\les |t|.
	\end{align*}

\end{theorem}

Estimates for the three-dimensional wave equation when
there is a resonance at zero energy were obtained by
Krieger and Schlag in \cite{KS}, and (with 
Nakanishi) extended to
mixed space-time estimates in \cite{KNS}.  In that work 
the authors
were concerned with the focusing, energy critical
non-linear wave equation in three spatial dimensions,
in which the dispersive evolution (after projecting
away from the resonance function) was needed.  During
the review period for this article, 
Erdo\smash{\u{g}}an, Goldberg the author  
proved dispersive 
estimates for the four-dimensional Schr\"odinger and
wave equations with resonances at zero energy, 
\cite{EGG}.

The resolvent operator 
$R_V(z):=(-\Delta+V-z)^{-1}$, 
is defined for $z\in \mathbb C$ which are not in the
spectrum of $H$.  We define the limiting operators
$$R_V^{\pm}(\lambda^2)
:=\lim_{\epsilon \to 0^+}(-\Delta+V-(\lambda^2\pm i\epsilon))^{-1}.$$
Such operators
are well defined on certain weighted $L^2$ spaces due to
Agmon's limiting absorption principle, see
\cite{agmon}.

Using the Spectral theorem we can write
$$
	\langle f(H)P_{ac} g, h\rangle=\int_0^\infty f(\lambda) 
	\la E_{ac}'(\lambda)g,h \ra \, d\lambda,
$$
where the Stone formula yields that the
absolutely continuous spectral measure is
given by 
$$
	\la E_{ac}'(\lambda)g,h \ra =\frac{1}{2\pi i}
	\la [R_V^+(\lambda)-R_V^-(\lambda)]g,h\ra.
$$
The key insight here is that the
absolutely continuous spectral measure
is the same as in the analysis of the Schr\"odinger
equation with potential.  The fact that the spectral measure is the
difference of the resolvent operators is crucial to
establishing dispersive bounds in the free case, and
for low energy (small $\lambda$) bounds we establish
in Section~\ref{sec:low energy}.
For low energy, the difference of the resolvents helps
to control the singularities of the resolvents as the
spectral parameter $\lambda\to 0$.
We write the evolution of the
solution operator projection onto the absolutely
continuous spectrum as
\begin{align}\label{spectral rep}
	\cos(t\sqrt{H})P_{ac}f(x)
	+\frac{\sin(t\sqrt{H})}{\sqrt{H}}P_{ac}g(x)
	&=\frac{1}{\pi i}
	\int_0^\infty \lambda \cos(t\lambda)
	\big[R_V^+(\lambda^2)-R_V^-(\lambda^2)
	\big]f(x)\, d\lambda\\
	&+\frac{1}{\pi i}
	\int_0^\infty \sin(t\lambda))
	\big[R_V^+(\lambda^2)-R_V^-(\lambda^2)
	\big]g(x)\, d\lambda.\nn
\end{align}
Where we made the change of variables 
$\lambda \mapsto \lambda^2$ for computational convience.
The solution operators in the wave
equation lead to multiplication in the spectral
parameter by the functions
sine and cosine instead of multiplication by
$e^{it\lambda^2}\lambda$ that arises in the Schr\"odinger
evolution.  Throughout this paper, we will consider the
case of $t>0$, the case of $t<0$ follows from minor
modifications of our proofs.

The dispersive decay, in the sense of mappings
between weighted $L^2$ spaces
or from $L^1$ to $L^\infty$ is well
studied for the perturbed Schr\"odinger equation
$$
	i\partial_t u-\Delta u+Vu=0, \qquad u(x,0)=f(x).
$$
In two spatial dimensions, 
Murata attained time-integrable estimates on weighted
$L^2$ spaces.  This time-decay rate was attained by
Erdo\smash{\u{g}}an and the author in a logarithmically
weighted $L^1$ to logarithmically weighted $L^\infty$ 
space in \cite{EG2}.  This paper proves theorems 
for the wave equation which are analogous
to those found in the work of
Schlag in \cite{Sc2} and previous  work of 
Erdo\smash{\u{g}}an and the author, \cite{EG,EG2} 
studying the Schr\"odinger solution operator
$e^{itH}P_{ac}$ in two spatial 
dimensions.


The paper is organized as
follows:  In Section~\ref{sec:exp} we develop expansions
for the two-dimensional free resolvent 
needed for
our analysis.  In Section~\ref{sec:high energy} we
establish the high-energy dispersive bound for 
Theorem~\ref{thm:main}.  
In Section~\ref{sec:low energy} we recall necessary
resolvent expansions for small $\lambda$ and establish
the dispersive bound of Theorem~\ref{thm:main}
for low energy.
In Section~\ref{sec:weighted} we establish refinements
of the bounds and expansions of the previous sections
needed to prove Theorem~\ref{thm:wtd}.
Finally
in Section~\ref{sec:nonreg} we show how the low-energy
evolution is affected by the presence of an eigenvalue
and/or resonance at zero energy, and prove Theorem~\ref{thm:nonreg}.

\section{The Free Resolvent} \label{sec:exp}

In this section we discuss the properties of the free resolvent,  $R_0^\pm(\lambda^2)=[-\Delta-(\lambda^2\pm i 0)]^{-1}$, in $\R^2$.  We wish to understand the
free resolvent in order to better understand the spectral
measure in \eqref{spectral rep}. A similar discussion
appears in \cite{Sc2,EG,EG2}.

To simplify the formulas, we   use the notation
$$
f=\widetilde O(g)
$$
to denote
$$
\frac{d^j}{d\lambda^j} f = O\big(\frac{d^j}{d\lambda^j} g\big),\,\,\,\,\,j=0,1,2,3,...
$$
If the derivative bounds hold only for the first $k$ derivatives we  write $f=\widetilde O_k (g)$.

Recall that
\begin{align}\label{R0 def}
	R_0^\pm(\lambda^2)(x,y)=\pm\frac{i}{4} 
	H_0^\pm(\lambda|x-y|)=\pm\frac{i}{4} 
	J_0(\lambda|x-y|)- \frac{1}{4} Y_0(\lambda|x-y|).
\end{align}
Thus, we have
\be\label{r0low2}
  R_0^+(\lambda^2)(x,y)-R_0^-(\lambda^2)(x,y) =  \frac{i}2 J_0(\lambda |x-y|).
\ee
From the series expansions for the Bessel functions, see \cite{AS},  we have, as $z\to 0$,
\begin{align}
	J_0(z)&=1-\frac{1}{4}z^2+\widetilde O_4(z^4),\label{J0 def}\\
	Y_0(z)&=\frac{2}{\pi}(\log(z/2)+\gamma)J_0(z)+\frac{2}{\pi}\bigg(\frac{1}{4}z^2 +\widetilde O_4(z^4)\bigg)\nn \\
&=\frac2\pi \log(z/2)+\f{2\gamma}{\pi}+ \widetilde O(z^2\log(z)).\label{Y0 def}
\end{align}
For $|z|>1 $,
\begin{align}\label{JYasymp2}
	 H_0^\pm(z)= e^{\pm i z} \omega_\pm(z),\,\,\,\,\quad \omega_{\pm}(z)=\widetilde O\big((1+|z|)^{-\frac{1}{2}}\big).
\end{align}
So that for $|z|>1$
\begin{align}\label{largeJYH}
	\mathcal C(z)=e^{iz} \omega_+(z)+e^{-iz}\omega_-(z), \qquad
	\omega_{\pm}(z)=\widetilde O\big((1+|z|)^{-\frac{1}{2}}\big),
\end{align}
 for any $\mathcal C\in\{J_0, Y_0\}$ respectively with different $\omega_{\pm}$ that satisfy the same bounds.

In particular, 
\begin{align}\label{r0low}
 R_0^\pm(\lambda^2)(x,y)&= \pm\frac{i}4-\frac{\gamma}{2\pi}-\frac{1}{2\pi}\log(\lambda|x-y|/2)+\widetilde O\big(\lambda^2|x-y|^2\log(\lambda|x-y|)\big),
 \quad \lambda|x-y|\ll 1.\\
	\label{r0high}
 R_0^\pm(\lambda^2)(x,y)& = e^{\pm i\lambda|x-y|}\omega_{\pm}(\lambda|x-y|), 
 \qquad \lambda|x-y|\gtrsim 1.
\end{align}

\section{High Energy}\label{sec:high energy}
Here we show high energy, $\lambda\geq \lambda_1>0$
for some fixed $\lambda_1>0$, estimates
for the evolution of the solution of \eqref{wave eqn} 
in \eqref{spectral rep}.  
The exact value of $\lambda_1>0$ is unimportant for the
following analysis, it will be determined by certain
resolvent expansions used in the low energy regime,
see Lemma~\ref{Minverse} below.

Take a smooth cut-off function $\chi\in C^\infty((0,\infty))$
with $\chi(x)=1$ if $x<\f12$ and $\chi(x)=0$ if
$x>1$.  Define $\widetilde \chi=1-\chi$, and
$\chi_1(x)=\chi(x/\lambda_1)$ with $\widetilde \chi_1
=1-\chi_1$.
For the high energy evolution, we prove the following
high energy version of Theorem~\ref{thm:main}.

\begin{prop}\label{high prop}

	For fixed $\lambda_1>0$, 
	if $|V(x)|\les \la x\ra^{-2-}$,
	we have the bounds
	\begin{align*}
		\|\cos(t\sqrt{H})H^{-\f34-}
		\widetilde\chi_{1}
		(\sqrt{H})\|_{L^1\to L^\infty}&\les |t|^{-\f12},\\
		\|\sin(t\sqrt{H})H^{-\f34-}
		\widetilde\chi_{1}
		(\sqrt{H})\|_{L^1\to L^\infty}&\les |t|^{-\f12}.
	\end{align*}

\end{prop}

We define the cut-off $\chi_L(\lambda)=\chi(\lambda/L)$.
Due to the spectral respresentation, \eqref{spectral rep},
the Proposition follows if we prove the 
bound
\begin{align}\label{weighteddecayhi}
	\sup_{x,y\in \R^2} \sup_{L\geq 1}
	\bigg| \int_{\R^2}
	\int_0^\infty e^{it\lambda}\lambda^{-\f12-} 
	\widetilde{\chi}_1(\lambda) \chi_L(\lambda)
	[R_V^+(\lambda^2)-R_V^-(\lambda^2)](x,y)\, d\lambda
	 \bigg|
	\les |t|^{-\f12}.
\end{align}
We prove the estimates
with $e^{it\lambda}$, estimates for $e^{-it\lambda}$ and
hence $\sin(t\lambda)$  and $\cos(t\lambda)$ follow
similarly.

We note that such a bound is attained in \cite{Mou} 
for $\lambda_1$ large enough with the
assumption that 
$$
	\sup_{y\in \R^2}\int_{\R^2}\frac{|V(z)|}
	{|z-y|^{\f12}}\,
	dz\leq C<\infty.
$$
As one needs more decay on the potential to understand the
low energy evolution and we need the bound to hold for
any $\lambda_1>0$,
we provide the following proof which requires a slightly
stronger assumption on the potential.

We start with the resolvent expansion
\begin{align}\label{BS:high}
    R_V^{\pm}(\lambda^2)&=
    R_0^{\pm}(\lambda^2)
    -R_0^{\pm}(\lambda^2)VR_0^{\pm}(\lambda^2) 
    +R_0^{\pm}(\lambda^2)VR_V^{\pm}(\lambda^2) V R_0^{\pm}(\lambda^2).
\end{align}
The first term is the
free resolvent.  In the high energy regime 
for the first term only, 
we need to use the cancellation between
`+' and `-' free resolvents, for the remaining terms the
cancellation between `+' and `-' resolvents does not
simplify the proofs and is not used.

\begin{lemma}\label{free dispersive}

	We have the bound
	\begin{align*}
		\sup_{x,y\in \R^2} \sup_{L\geq 1}
		\bigg|
		\int_0^\infty e^{it\lambda}\lambda^{-\f12-} 
		\widetilde\chi_1(\lambda) \chi_L(\lambda)
		[R_0^{+}(\lambda^2)-R_0^{-}(\lambda^2)]
		(x,y)\,d\lambda
		\bigg| \les t^{-\f12}.
	\end{align*}

\end{lemma}

\begin{proof}

Using \eqref{r0low2}, we reduce the contribution of
the first term in \eqref{BS:high} to showing the bound
	$$
		\sup_{x,y\in \R^2}\sup_{L\geq 1}
		\bigg|\int_0^\infty e^{it\lambda} \widetilde \chi_1(\lambda) \chi_L(\lambda)
		\lambda^{-\f12-} J_0(\lambda |x-y|)\, d\lambda
		\bigg| \les t^{-\f12}.
	$$

Using \eqref{J0 def} and \eqref{JYasymp2} we write
$$
	J_0(\lambda |x-y|)=\rho_-(\lambda |x-y|)
	+e^{i\lambda |x-y|}\omega_+(\lambda |x-y|)
	+e^{-i\lambda |x-y|}\omega_-(\lambda |x-y|)
$$
with $\rho_-(z)=\widetilde O(1)$ supported on
$[0,\f12]$
and $\omega_{\pm}(z)=\widetilde O((1+z)^{-1/2})$
supported on $[\f14,\infty)$.

We first consider the  contribution of $J_0$ when 
$\lambda|x-y|\gtrsim 1$ we consider first the `-'
phase from the expansion for $J_0$ and bound
the following integral
\begin{align}\label{high:free high}
	\int_0^\infty e^{it\lambda} \widetilde \chi_1(\lambda)
	\chi_L(\lambda)
	\lambda^{-\f12-}	 e^{-i\lambda|x-y|}\omega_-(\lambda|x-y|)
	\, d\lambda.
\end{align}
First we note that if $t-|x-y|\leq t/2$ then $t\les |x-y|$
and we use the decay of $\omega_-$ to see
\begin{align*}
	|\eqref{high:free high}|
	&\les \int_1^\infty \lambda^{-1-} |x-y|^{-\f12-}\, d\lambda
	\les |x-y|^{-\f12-}\les t^{-\f12-}.
\end{align*}
On the other hand, if $t-|x-y|\geq t/2$, we can integrate
by parts against the combined phase $e^{i\lambda(t-|x-y|)}$
to see
\begin{align*}
	|\eqref{high:free high}|
	&\les \frac{1}{t-|x-y|}\int_0^\infty
	\frac{d}{d\lambda}\big(\widetilde \chi_1(\lambda)\chi_L(\lambda)\lambda^{-\f12-}
	\omega_-(\lambda|x-y|)	\big)\, d\lambda
	\les \frac{1}{t}\int_1^\infty 
	\frac{\lambda^{-\f32-}}{(1+\lambda |x-y|)^{\f12}}
	\, d\lambda\\
	&\les \frac{1}{t}|x-y|^{-\f12}
	\int_{\lambda \gtrsim |x-y|^{-1}} \lambda^{-2}\, d\lambda\les \frac{|x-y|^{\f12}}{t}
	\les t^{-\f12}.
\end{align*}
Where we used that $\lambda\gtrsim 1$ in the second
to last step and that $|x-y|\les t$ in the last 
inequality.  For the `+' phase, the analysis follows
from considering the cases of $t+|x-y|\leq 2t$
and $t+|x-y|\geq 2t$ in place of the cases
$t-|x-y|\geq t/2$ and $t-|x-y|\leq t/2$ considered for
the negative phase.

For the contribution of $\rho_-$ we need to consider
three cases.  First, consider the case when $t\les |x-y|$.
Then, using $1\les (\lambda |x-y|)^{-1}$ on the support of
$\rho_-$ we have
\begin{align*}
	\int_0^\infty e^{it\lambda} \widetilde \chi_1(\lambda)
	\chi_L(\lambda)
	\lambda^{-\f12-} \chi_{\{\lambda|x-y|\les 1\}} \, d\lambda
	&\les \frac{1}{|x-y|^{\f12}}
	\int_{0}^\infty  \lambda^{-1-}\widetilde \chi_1
	(\lambda)
	\, d\lambda \les \frac{1}{|x-y|^{-\f12}}\les t^{-\f12}.
\end{align*}

We now consider the second case in which $t\gtrsim |x-y|$.
The first subcase here is when $t\gtrsim 1$.
We use the fact that
$|\frac{d^k}{d\lambda^k} \widetilde \chi_1(\lambda)|,
|\frac{d^k}{d\lambda^k} \widetilde \chi_L(\lambda)|
\les \lambda^{-k}$ to see
\begin{align*}
	\int_0^\infty e^{it\lambda} \widetilde \chi_1(\lambda)
	\chi_L(\lambda)
	\lambda^{-\f12-} \widetilde O(1)\, d\lambda
	&=\frac{1}{it}\int_0^\infty e^{it\lambda} \frac{d}
	{d\lambda}\big(\widetilde \chi_1(\lambda)
	\chi_L(\lambda)
	\lambda^{-\f12-} \widetilde O(1)\big )\, d\lambda\\
	&\les \frac{1}{t} \int_1^\infty  \lambda^{-\f32-}\, 
	d\lambda \les t^{-1}\les t^{-\f12}.
\end{align*}
 Finally, we need to consider the case in which
 $|x-y|\ll t\ll 1$.  Then we break the integral up into
 two pieces to bound with
 \begin{align*}
 	\int_0^{t^{-1}} e^{it\lambda} \widetilde \chi_1(\lambda)\chi_L(\lambda)
 	\lambda^{-\f12-} \widetilde O(1)\, d\lambda
 	+ 	\int_{t^{-1}}^\infty e^{it\lambda} \widetilde \chi_1(\lambda)\chi_L(\lambda)
 	 \lambda^{-\f12-} \widetilde O(1)\, d\lambda.
 \end{align*}
 For the first integral we note that on the support of
 $\widetilde \chi_1(\lambda)$, we have $\lambda^{-\f12 -}
 \les \lambda ^{-\f12}$ so,
  \begin{align*}
  	\bigg|\int_0^{t^{-1}} e^{it\lambda} \widetilde \chi_1(\lambda)\chi_L(\lambda)
  	\lambda^{-\f12-} \widetilde O(1)\, d\lambda\bigg|
  	&\les \int_0^{t^{-1}} \lambda^{-\f12}\, d\lambda
  	\les t^{-\f12}
  \end{align*}
  For the second integral, we integrate by parts once
  against the phase $e^{it\lambda }$ to bound with 
\begin{align*}
	\int_{t^{-1}}^\infty e^{it\lambda} \widetilde 
	\chi_1(\lambda)\chi_L(\lambda)
	\lambda^{-\f12-} \widetilde O(1)\, d\lambda&\les
	\frac{\widetilde \chi_1(\lambda)\chi_L(\lambda)
	\lambda^{-\f12-}}{t} \bigg|_{t^{-1}}^\infty +\frac{1}{t} \int_{t^{-1}}^\infty\lambda^{-\f32 -}\, d\lambda
	\les t^{-\f12}.
\end{align*}
Again we use that
that on the support of
 $\widetilde \chi_1(\lambda)$, we have $\lambda^{-\f12 -}
 \les \lambda ^{-\f12}$ to get the $t^{-\f12}$ bound.

\end{proof}

It is worth noting one can obtain this high energy
estimate by considering an alternative definition of
the Bessel function to bound
$$
	\int_0^\infty e^{it\lambda}
	\lambda^{-\f12-}\widetilde{\chi}_1(\lambda)
	\chi_L(\lambda)\bigg(
	\int_{|\omega|=1} e^{i\lambda \omega \cdot (x-y)}\, d\omega
	\bigg)\, d\lambda.
$$
The desired bounds follow by using standard facts on
the Fourier transforms of densities on curved manifolds,
see \cite{soggebook} for example.
We do not pursue this issue here, as
much of our analysis requires considering the free
resolvent and the Bessel functions on different
regimes according to the size of their argument, 
recall \eqref{J0 def}, \eqref{Y0 def},
\eqref{JYasymp2}, \eqref{r0low} and \eqref{r0high}.
This is due to the vastly different behavior of the
`low', $|z|\ll 1$ and `high', $|z|\gtrsim 1$, portions 
of these functions.

For the second term in \eqref{BS:high}, we have

\begin{lemma}\label{lem:bs2hi}

	If $|V(x)|\les \la x\ra^{-2-}$, then we 
	have the bound
	\begin{align*}
		\sup_{x,y\in \R^2} \sup_{L\geq 1}
		\bigg| \int_{\R^2}
		\int_0^\infty e^{it\lambda}\lambda^{-\f12-} 
		\widetilde\chi_1(\lambda) \chi_L(\lambda)
		R_0^{\pm}(\lambda^2)(x,z)V(z)
		R_0^{\pm}(\lambda^2)(z,y)\,d\lambda\, dz
		\bigg| \les t^{-\f12}.
	\end{align*}

\end{lemma}

\begin{proof}

We won't make use of any  cancellation between `$\pm$' terms. Thus, we will only consider $R_0^-$, and drop the `$\pm $' signs.
Using \eqref{R0 def}, \eqref{J0 def}, \eqref{Y0 def}, and \eqref{JYasymp2} we write
\begin{align}\label{Resolvent decomp}
    R_0(\lambda^2)(x,y)=e^{-i\lambda|x-y|}\rho_+(\lambda|x-y|)+\rho_-(\lambda|x-y|),
\end{align}
where $\rho_+$ and $\rho_-$ are supported on the sets $[1/4,\infty)$ and $[0, 1/2]$, respectively.
Moreover, we have the bounds
\be\label{omega bds}
    \rho_-(y) =\widetilde O(1+|\log y|),\,\,\,\,\,\,\, \rho_+(y) =\widetilde O\big( (1+|y|)^{-1/2}\big)\\
\ee

We now consider
$$
	\int_0^\infty e^{it\lambda}\lambda^{-\f12-} 
	\widetilde{\chi}_1(\lambda)\chi_L(\lambda) R_0^{-}(\lambda^2)(x,z)V(z)
	R_0^{-}(\lambda^2)(z,y)\,d\lambda.
$$
We are considering the case of $R_0^{-}$ as the sign of
$t$ and the high energy phases do not match, which requires
a more delicate analysis.
We note that
\begin{align*}
\big|\partial_\lambda^k \big[\rho_+(\lambda d_j)\big]\big|&\les \frac{d_j^k}{(1+\lambda d_j)^{k+1/2}},\,\,\,\,\,\,\,k=0,1,2,...,\\
\big|\partial_\lambda^k \big[\rho_-(\lambda d_j)\big]\big|&\les \f1{\lambda^k},\,\,\,\,k=1,2,...
\end{align*}
Define the function $\log^{-}(z)=- (\log z)\chi_{\{0<z<1\}}$.
Using the monotonicity of $\log^-$ function, we also obtain
$$
\widetilde\chi_1(\lambda) \big|\rho_-(\lambda d_j)\big|\les \widetilde \chi_1(\lambda) (1+|\log(\lambda d_j)|) \chi_{\{0<\lambda d_j\leq 1/2\}}
\les \widetilde \chi(\lambda) (1+\log^-(\lambda d_j))\les
1+\log^-(d_j).
$$
Finally, noting that
 $(\widetilde \chi_1)^{\prime}$ and $\chi_L^\prime$
 are supported on the set $\{\lambda\approx 1\}$,
 we have
 \begin{align}\label{log-eqn}
 	|\partial_\lambda^k \widetilde \chi_1(\lambda)|,
 	\quad
 	|\partial_\lambda^k \chi_L(\lambda)| \les \lambda^{-k}.
 \end{align}

For notational convience, let $d_1:=|x-z|$ and
$d_2:=|z-y|$.
First consider the `low-low' interaction, that is when each
resolvent contributes $\rho_-$.    
We consider two cases,
first if $t\geq \sqrt{d_1d_2}$, then
\begin{align*}
	\int_0^\infty e^{it\lambda}
	\widetilde \chi_1(\lambda)& \chi_L(\lambda)
	\lambda^{-\f12-} 
	\rho_{-}(\lambda d_1)
	\rho_-(\lambda d_2)\,d\lambda\\
	&\les \frac{1}{t} \int_0^\infty \bigg|
	\frac{d}{d\lambda}\big[
	\widetilde \chi_1(\lambda)\chi_L(\lambda)
	\lambda^{-\f12-} 
	\rho_{-}(\lambda d_1)
	\rho_-(\lambda d_2)\big] \bigg|\,d\lambda\\
	&\les \frac{(1+\log^-d_1)(1+\log^-d_2)}{t}
	\int_1^\infty \frac{1}
	{\lambda^{\f32+} }\, d\lambda
	\les \frac{(1+\log^-d_1)(1+\log^-d_2)}
	{t^{\f12}d_1^{\f14}d_2^{\f14}}.
\end{align*}
On the other hand, if $t\leq \sqrt{d_1 d_2}$ we don't
integrate by parts and instead use the support conditions
on $\rho_-$ to see
\begin{align}\label{lowlow:noIBP}
	\int_0^\infty e^{it\lambda}
	\widetilde \chi_1(\lambda)\chi_L(\lambda)
	\lambda^{-\f12-} 
	\rho_{-}(\lambda d_1)
	\rho_-(\lambda d_2)\,d\lambda
	&\les \int_1^\infty \frac{\lambda^{-\f12-}}{(\lambda \sqrt{d_1d_2})}\, d\lambda
	\les \frac{1}{t^{\f12}}\frac{1}{d_1^{\f14}d_2^{\f14}}.
\end{align}

Next we consider the `low-high' interaction.  That is, we
have a contribution from one $\rho_+$ and one $\rho_-$.
Accordingly, we need to bound integrals of the form
\begin{align}\label{int high-high-high}
	\int_0^\infty e^{it\lambda}
	\widetilde \chi_1(\lambda)\lambda^{-
	\f12-} 
	e^{-i\lambda d_1}\rho_{+}(\lambda d_1)
	\rho_-(\lambda d_2)\,d\lambda
\end{align}
Due to the a-priori $\lambda$ decay of $\rho_+$,
\eqref{JYasymp2}, we can take $L=\infty$ and remove the
cut-off from our analysis.
We need to consider two cases, depending on the relative size
of $t$ and the phase $d_1$.  In the first case, we consider
$t-d_1\geq \frac{t}{2}$.  In this case we 
can safely integrate by parts
against $e^{i\lambda(t-d_1)}$. 
\begin{align*}
	|\eqref{int high-high-high}| &\les
	\frac{1}{t-d_1}\int_0^\infty \bigg| \frac{d}{d\lambda}
	\big(\widetilde \chi_1(\lambda)
	\lambda^{-\f12-} \rho_+(\lambda d_1)
	\rho_-(\lambda d_2)\big)\bigg|	\, d\lambda\\
	&\les \frac{1}{t}\int_1^\infty \bigg[
	\frac{1+\log^- d_2}{\lambda^{\f32+}(1+\lambda d_1)^{\f12}}+
	\frac{ d_1(1+\log^- d_2)}
	{\lambda^{\f12+}(1+\lambda d_1)^{\f32}}
	\bigg]\, d\lambda\\
	&\les \frac{1+\log^- d_2}{t}\int_{1}^\infty 
	\lambda^{-\f32-}\, d\lambda
	\les \frac{1+\log^- d_2}{t} 
	\les \frac{1+\log^- d_2}{t^{\f12}d_1^{\f12}}
\end{align*}
Here we used that $t^{-1}\les d_1^{-1}$ in the last
inequality.

The second case is when $t-d_1\leq \f t2$.  In this case we
won't integrate by parts, but use that $d_1^{-1}\les t^{-1}$.
\begin{align*}
	|\eqref{int high-high-high}|
	&\les \int_1^\infty 
	\frac{\lambda^{-\f12-}(1+\log^- d_2)}{(1+\lambda d_1)^{\f12}}\, d\lambda
	\les \frac{1+\log^- d_2}{d_1^{\f12}} 
	\int_{1}^\infty \lambda^{-1-}\, d\lambda
	\les \frac{1+\log^- d_2}{t^{\f12}}.
\end{align*}
Note that it is precisely this case that constrains us
to a $t^{-\f12}$ time decay rate in the wave equation
for a non-weighted $L^\infty$ bound.

For the `high-high' interaction, we need to bound
\begin{align}\label{hhh int}
	\int_0^\infty e^{it\lambda}
	\widetilde \chi_1(\lambda)\lambda^{-\f12-} 
	e^{-i\lambda (d_1+d_2)}\rho_{+}(\lambda d_1)
	\rho_+(\lambda d_2)\,d\lambda.
\end{align}
We first consider the case when $t-(d_1+d_2)\leq \f t2$,
in which case $(d_1+d_2)\geq \f t2$.  Then
\begin{align*}
	|\eqref{hhh int}|&\les \frac{1}{\sqrt{d_1 d_2}}
	\int_1^\infty 
	\lambda^{-\f32-} \,d\lambda \les \frac{1}{\sqrt{d_1 d_2}}
	\les \frac{1}{t^{\f12}\sqrt{\min(d_1, d_2)}}
\end{align*}
Where we used that $(d_1+d_2)\geq \f t2$ gives us that
$\max(d_1,d_2)\gtrsim t$.  

The second case is when $t-(d_1+d_2)\geq \f t2$, in which
we case we don't integrate by parts once but use 
$d_1, d_2\les t$ and the asymptotics of $\rho_+$.  Accordingly we bound
\begin{align*}
	\frac{1}{\sqrt{d_1d_2}}
	\int_0^\infty 
	\widetilde \chi_1(\lambda) \lambda^{-\f32-}
	\rho_+(\lambda d_1)\rho_+(\lambda d_2)
	\, d\lambda
	&\les\frac{1}{\sqrt{d_1d_2}}
	\int_1^\infty \lambda^{-\f32-}\, d\lambda
	\les \frac{1}{t^{\f12}\sqrt{\min(d_1, d_2)}}
\end{align*}
The case when $R_0^+$ appears and the exponentials
have a positive phase follows as in the case analysis
above by considering the size of $t+d_1$ or
$t+d_1+d_2$ compared to $2t$.

Finally, we close the argument by noting that
$$
	\sup_{x,y} \bigg|\int_{\R^2} V(z)\bigg(
	1+\frac{(1+\log^-|z-y|)(1+\log^-|x-z|)}
	{|x-z|^{\f14}|z-y|^{\f14}}+
	\frac{1+\log^-|z-y|}{|x-z|^{\f12}}
	\bigg)\, dz \bigg|\les 1
$$
provided that $|V(z)|\les \la z\ra^{-2-}$.

\end{proof}

To control the remainder of the born series
we employ the limiting absorption principle of Agmon,
\cite{agmon}. For $\lambda>\lambda_1>0$
\begin{align}\label{free lap}
	\| R_0^{\pm}(\lambda^2)\|_{L^{2,\sigma}(\R^2)\to L^{2,-\sigma}(\R^2)}\les \lambda^{-1+}.
\end{align}
From the limiting absorption principle, we can
deduce  bounds for
the derivatives by considering the respresentation
for the free resolvent in \eqref{Resolvent decomp}.
Specifically, we have
\begin{align}\label{lap}
	\| R_V^{\pm}(\lambda^2)\|_{L^{2,\sigma}(\R^2)\to L^{2,-\sigma}(\R^2)}&\les \lambda^{-1+},\\
	\sup_{\lambda>\lambda_1}\| \partial_\lambda^k R_V^{\pm}(\lambda^2)\|_{L^{2,\sigma}(\R^2)\to L^{2,-\sigma}(\R^2)}&\les 1,
\end{align}
which is valid for $\sigma>\frac{1}{2}+k$.  The bounds for
the derivatives are valid for the free resolvents as well.
The proof of these estimates follows as in Proposition~9
in \cite{GS} along with the discussion following it.
The bounds for the perturbed
resolvent  relies on the resolvent identity
$$
	R_V^\pm (\lambda^2)=(I +R_0^\pm(\lambda^2)V)^{-1}
	R_0^\pm(\lambda^2)
$$
and the absence of embedded
eigenvalues in the continuous spectrum, which is guaranteed
by Kato's Theorem, see Section XIII.8 of \cite{RS1}.

 Using the representation \eqref{omega bds},
we note the following bounds on the free resolvent
which are valid on $\lambda>\lambda_1>0$,
\begin{align}\label{r0bds}
	|\partial_\lambda^k R_0^{\pm}(\lambda^2)(x,y)|\les |x-y|^k \left\{\begin{array}{ll}
	|\log(\lambda|x-y|)| & 0<\lambda|x-y|<\frac{1}{2}\\
	(\lambda|x-y|)^{-\frac{1}{2}} & \lambda|x-y|\gtrsim 1\end{array}\right.
	\les \lambda^{-\frac{1}{2}}|x-y|^{k-\frac{1}{2}}.
\end{align}
Thus, for $\sigma>\f12+k$,
\begin{align}\label{R0 wtdL2}
	\|\partial_\lambda^k &R_0^{\pm}(\lambda^2)(x,y)\la y\ra^{-\sigma}\|_{L^2_y} 
	\les\lambda^{-\f12} \Big[\int_{\R^2} \frac{|x-y|^{2k-1}}{\la y\ra^{2\sigma}}\, dy\Big]^{\f12}
	\les \lambda^{-\f12} \la x\ra^{\max(0,k-\f12)}.
\end{align}
To avoid polynomial weights in $x$ or $y$, 
we take a bit more care with the
leading and lagging free resolvents.  We use the following
estimates on the pieces of the free resolvent in
\eqref{Resolvent decomp}.  Noting that 
for $\lambda\gtrsim 1$,
\begin{align}
	|\partial_\lambda^k \rho_-(\lambda)(\lambda r)|
	&\les (1+\lambda r)^{-\f12}
	\left\{\begin{array}{ll} 
	(1+ \log^- r) & k=0\\
	\lambda^{-1} & k=1
	\end{array}\right.\\
	|\partial_\lambda^k \rho_+(\lambda)(\lambda r)|
	&\les \lambda^{-k} (1+\lambda r)^{-\f12}, \qquad
	k=0,1.
\end{align}
The $(1+\lambda r)^{-\f12}$ in the first bound follows
from the support conditions on $\rho_-$.
So that for $\sigma>k+\frac{1}{2}$ and $k=0,1$
we have the weighted $L^2$ bounds
\begin{align}\label{rho L2}
	\| \partial_\lambda^k \rho_{\pm}(\lambda|x-y|)
	\|_{L^{2,\sigma}_y}\les \lambda^{-k-\frac{1}{2}}.
\end{align}

Once again, we estimate the $R_V^+$ and $R_V^-$ terms 
separately and omit the `$\pm$' signs. 

\begin{lemma}\label{lem:tailhi}

If $|V(x)|\les \la x\ra^{-2-}$ we have the bound
\eqref{weighteddecayhi} as
	\begin{align}\label{I def}
		\sup_{x,y\in \R^2} \sup_{L\geq 1}
		 \bigg|\int_0^\infty e^{it\lambda} \lambda^{-\f12-} \, \mathcal E(\lambda)(x,y)\, d\lambda\bigg|
		 \les t^{-\f12}
	\end{align}
where
$$
	\mathcal E(\lambda)(x,y) = 
	\widetilde{\chi}_1(\lambda) \chi_L(\lambda)
	\big\la VR_V^{\pm}(\lambda^2) V 
	R_0^{\pm}(\lambda^2)(\cdot,x),
 	R_0^{\pm}(\lambda^2)(\cdot,y) \big\ra.
$$

\end{lemma}

\begin{proof}

Using \eqref{R0 wtdL2}, one can see that
\begin{align*}
	|\eqref{I def}|&\les \int_0^{\infty} 
	\widetilde\chi_1(\lambda)\lambda^{-\f12-}
	\|R_0\|_{L^{2,-\f12-}} \|V\|_{L^{2,-\f12-}\to L^{2,\f12+}}
	\|R_V\|_{L^{2,\f12+}\to L^{2,-\f12-}}\\ 
	&\qquad \times \|V\|_{L^{2,-\f12-}\to L^{2,\f12+}}
	\|R_0\|_{L^{2,-\f12-}}\, d\lambda
	\les \int_0^{\infty}\widetilde \chi_1(\lambda)\lambda^{-\f32-}\, d\lambda \les 1.
\end{align*}
Again this justifies taking $L=\infty$ in the cut-off
$\chi_L$.
To establish bounds with time decay, 
we again need to distinguish cases based on whether
the leading and lagging resolvents contribute $\rho_-$
or $\rho_+$ and the size of $t$ relative to $|x-z|$, 
and/or $|y-z|$.  The analysis is essentially identical
to the second term of the Born Series, but here we 
use weighted $L^2$ bounds.  For example, we can conclude
the bound for the term analogous to \eqref{hhh int}
by noting that
\begin{align*}
	\|1+\log^- |x-\cdot|\|_{L^{2,-1-}} \|V\|_{L^{2,-1-}\to L^{2,\f12+}}
	\|R_V^{\pm}\|_{L^{2,\f12+}\to L^{2,-\f12-}} 
	\|V\|_{L^{2,-\f12-}\to L^{2,\f12+}}
	\||y-\cdot|^{-\f12}\|_{L^{2,-\f12-}}
\end{align*}
is bounded uniformly in $x$ and $y$ under the assumptions
on the potential.

We need only make a few small adjustments.  We note that, 
when integrating by parts, if the derivative acts on the
resolvent $R_V^{\pm}(\lambda^2)$ in the case when only
$\rho_-$'s are contributing 
we must account for the case when the derivative
acts on the perturbed resolvent,
\begin{multline*}
	\int_0^\infty \widetilde \chi_1(\lambda)
	\lambda^{-\f12-}\rho_-(\lambda d_1)\rho_-(\lambda d_2)
	\partial_{\lambda}R_V^{\pm}(\lambda^2)\, d\lambda\\
	\les \frac{1}{d_1^{\f14}d_2^{\f14}}\int_0^\infty \widetilde \chi_1(\lambda)
	\lambda^{-1-}\rho_-(\lambda d_1)\rho_-(\lambda d_2)
	\partial_{\lambda}R_V^{\pm}(\lambda^2)\, d\lambda
\end{multline*}
by using the support conditions that $\lambda d_j\les 1$.
Then, using \eqref{r0bds} on can 
see $\|r^{-\f14} \rho_-(\lambda r)\|_{L^{2,\sigma}}\les 1$ 
for $\sigma>\f12$.

	This argument requires that $|V(x)|\les \la x\ra^{-2-}$. One can see that the requirement on the decay rate of the potential arises when, for instance,
	the $\lambda$ derivative act on one of the
	$\rho_{\pm}$'s, this differentiated object is in $L^{2,-\frac{3}{2}-}$ by \eqref{rho L2}.
	The potential then needs to map
	$L^{2,-\frac{3}{2}-}\to L^{2,\frac{1}{2}+}$ for the  application of the limiting absorption principle for
	$R_V$.  On the other hand, if the derivative acts
	on the perturbed resolvent, $\partial_\lambda R_V$ 
	maps $L^{2,\frac{3}{2}+}\to L^{2,-\frac{3}{2}-}$ by \eqref{lap}, and $V$ must map $L^{2,-\f32-}$ to
	$L^{2,\f12+}$.

\end{proof}

We can now prove the high energy bound.

\begin{proof}[Proof of Proposition~\ref{high prop}]

	The statement follows from the spectral representation
	\eqref{weighteddecayhi},  and
	Lemmas~\ref{lem:bs2hi} and
	\ref{lem:tailhi} to control each term in 
	the Born series	expansion \eqref{BS:high}.

\end{proof}

\begin{rmk}\label{rmk:high}

	We note that one can prove estimates without the
	`regularizing' powers of $H$ by directly appealing
	to differentiability of the initial data.  That is,
	one can show
	\begin{align*}
		\sup_{x\in \R^2}
		\sup_{L\geq1}\bigg| \int_{\R^2}
		\int_0^\infty e^{it\lambda}\lambda 
		\widetilde{\chi}_1(\lambda) 
		\chi_L(\lambda)
		[R_V^+(\lambda^2)-R_V^-(\lambda^2)](x,y)f(y)\, d\lambda
		\, dy \bigg|
		\les |t|^{-\f12}\|f\|_{W^{2,1}}\\
		\sup_{x\in \R^2}
		\sup_{L\geq1}\bigg| \int_{\R^2} 
		\int_0^\infty e^{it\lambda} 
		\widetilde{\chi}_1(\lambda)\chi_L(\lambda)
		[R_V^+(\lambda^2)-R_V^-(\lambda^2)](x,y)f(y)\, d\lambda
		\, dy \bigg|
		\les |t|^{-\f12}\|f\|_{W^{1,1}}
	\end{align*}
	To do this, one can employ an integration by parts
	in the $y$ spatial variable according to
	\begin{align*}
		\int_{\R^2} e^{-i\lambda|z-y|}&\rho_+(\lambda|z-y|)f(y)
		\, dy
		= \frac{i}{\lambda}\int_{\R^2}e^{-i\lambda|z-y|}
		\nabla_y \cdot \Big[\rho_+(\lambda|z-y|)f(y)
		\frac{y-z}{|y-z|}\Big]\, dy\\
		&=-\frac{1}{\lambda^2}\int_{\R^2}
		e^{-i\lambda|z-y|}
		\nabla_y\cdot \Big\{\nabla_y \cdot \Big[\rho_+(\lambda|z-y|)f(y)
		\frac{y-z}{|y-z|}\Big] 
		\frac{y-z}{|y-z|}\Big\}\, dy.
	\end{align*}
	to gain $\lambda$ decay to ensure integrability
	as $\lambda \to \infty$ for the `high-high' 
	interactions.  One must also iterate the
	resolvent identity to form a longer Born series
	expansion than used in \eqref{BS:high},
	and use similar (but modified) 
	techniques as we used above.  This method, however,
	misses the sharp smoothness requirement by $\f12$
	of a derivative instead of the $\epsilon$ loss
	presented in this section.

\end{rmk}

\section{Low Energy}\label{sec:low energy}
To analyze the evolution of \eqref{wave soln} on low
energy, we must utilize different expansions for the
resolvent.  The low energy contribution to 
Theorem~\ref{thm:main} can again be reduced to
an integral bound due to \eqref{spectral rep}.

\begin{prop}\label{prop:lowbd}

	If $|V(x)|\les \la x\ra^{-3-}$ and if
	zero is a regular point of the spectrum of $H=-\Delta+V$, we have the bound
	\begin{align*}
		\sup_{x,y\in \R^2}
		\bigg|\int_0^\infty (\sin(t\lambda)
		+\lambda \cos(t\lambda)) 
		\chi_1(\lambda)[R_V^+(\lambda)(x,y)
		-R_V^-(\lambda)(x,y)]\, d\lambda\bigg|
		\les t^{-\f12}.
	\end{align*}

\end{prop}

First we state the needed resolvent
expansions and then control their contribution before
proving the dispersive bound.

\subsection{Low energy Resolvent expansions}

Let $U(x)=1$ if $V(x)\geq 0$ and $U(x)=-1$ if $V(x)<0$, and let $v=|V|^{1/2}$. We have $V=Uv^2$.
We use the symmetric resolvent identity, valid for $\Im\lambda>0$:
\be\label{res_exp}
R_V^\pm(\lambda^2)=  R_0^\pm(\lambda^2)-R_0^\pm(\lambda^2)vM^\pm(\lambda)^{-1}vR_0^\pm(\lambda^2),
\ee
where $M^\pm(\lambda)=U+vR_0^\pm(\lambda^2)v$.
The key issue studied in \cite{JN} and used in 
\cite{Sc2,EG,EG2} in the resolvent expansions is the invertibility of the operator $M^\pm(\lambda)$ for small $\lambda$
under various spectral assumptions at zero. 
For the sake of
brevity, we use the expansions of \cite{Sc2,EG,EG2}
as needed and omit the proofs.  To understand the
operator 
$M^{\pm}(\lambda)$ we define the following (see \eqref{Y0 def}), 
\begin{align}
	G_0f(x)&=-\frac{1}{2\pi}\int_{\R^2} \log|x-y|f(y)\,dy
	=(-\Delta)^{-1}f(x), \label{G0 def}\\
\label{g form}
		g^{\pm}(\lambda)&:= \|V\|_1\Big(\pm \frac{i}{4}-\frac{1}{2\pi}\log(\lambda/2)-\frac{\gamma}{2\pi} \Big).
\end{align}
We consider
more complicated expansions for $M^\pm(\lambda)$
in Sections~\ref{sec:weighted} and \ref{sec:nonreg}.

Our first expansion is a modification of Lemma~5 in
\cite{Sc2} which is Lemma~1 in \cite{EG2}.
\begin{lemma}\label{R0 exp cor}	
	We have the following expansion for the kernel of the free resolvent
	\begin{align*}
		R_0^{\pm}(\lambda^2)(x,y)=\frac{1}{\|V\|_1} g^{\pm}(\lambda)+G_0(x,y)
		+E_0^{\pm}(\lambda)(x,y).
	\end{align*}
	Here $G_0(x,y)$ is the kernel of the operator $G_0$ in \eqref{G0 def}, $g^{\pm}(\lambda)$ is as in \eqref{g form}, and
$E_0^{\pm}$
	satisfies the bounds
	\begin{align*}
		|E_0^{\pm}|\les \lambda^{\frac{1}{2} }|x-y|^{\frac{1}{2} }, \qquad
		|\partial_\lambda E_0^{\pm}|\les \lambda^{-\frac{1}{2} }|x-y|^{\frac{1}{2} }, \qquad|\partial_\lambda^2 E_0^{\pm}|\les \lambda^{-\frac{1}{2} }|x-y|^{\frac{3}{2}}.
		\end{align*}
\end{lemma}
We note that the expansion for the free resolvent we 
develop here appears to depend 
on the potential $V$ through the
factor of $\|V\|_1$ that appears in the functions
$g^{\pm}(\lambda)$.  We include this factor here to
simplify certain operators later on and note that our
goal is to develop an expansion for the operators
$M^{\pm}(\lambda)=U+vR_0^{\pm}(\lambda^2)v$ which do
explicitly depend on the potential.

This expansion can be interpreted as follows, the first
term $g^{\pm}(\lambda)$ contains all of the singularities
in $\lambda$, the second term is the integral kernel of
the Green's function
for $-\Delta$ and the remaining term is an
error term which is small in $\lambda$.  Unlike in 
higher dimensional cases, we cannot simply expand the
free resolvent as a small $\lambda$ perturbation of
the integral kernel operator $-(\Delta)^{-1}(x,y)$ due
to the logarithmic behavior of $g^{\pm}(\lambda)$.

We employ this
notation to match with that of previous works, 
see \cite{Sc2,EG,EG2}, in the context of the
Schr\"odinger equation.  
The error terms in the expansions
are not related to and should not be confused with the spectral family $E_{ac}'$
in the spectral theorem.
We note that the bound on the second derivative
of the error term is
not used here, but is needed to consider weighted
estimates in Section~\ref{sec:weighted}.

\begin{lemma} \label{lem:M_exp} 
	For $\lambda>0$ define $M^\pm(\lambda):=U+vR_0^\pm(\lambda^2)v$.
	Let $P=v\langle \cdot, v\rangle \|V\|_1^{-1}$ denote the orthogonal projection onto $v$.  Then
	\begin{align*}
		M^{\pm}(\lambda)=g^{\pm}(\lambda)P+T+E_1^{\pm}(\lambda).
	\end{align*}
	Here
	$T=U+vG_0v$ where $G_0$ is an
	integral operator defined in \eqref{G0 def}.
	Further, the error term satisfies the bound
	\begin{align*}
		\big\| \sup_{0<\lambda<\lambda_1} \lambda^{-\frac{1}{2}} |E_1^{\pm}(\lambda)|\big\|_{HS}
		+\big\| \sup_{0<\lambda<\lambda_1} \lambda^{\frac{1}{2}} |\partial_\lambda E_1^{\pm}(\lambda)|\big\|_{HS}		
		\les 1,
	\end{align*}
	provided that $v(x)\lesssim \langle x\rangle^{-\frac{3}{2}-}$.
	
\end{lemma}

\begin{proof}

	This follows from Lemma~\ref{R0 exp cor} the definition
	of the operators
	$$
		M^\pm (\lambda)^{-1}=U+vR_0^\pm(\lambda^2)v
	$$

\end{proof}

We recall the following  definition from \cite{Sc2} and \cite{EG}.
\begin{defin}
	We say an operator $T:L^2(\R^2)\to L^2(\R^2)$ with kernel
	$T(\cdot,\cdot)$ is absolutely bounded if the operator with kernel
	$|T(\cdot,\cdot)|$ is bounded from $L^2(\R^2)$ to $L^2(\R^2)$.
\end{defin}
It is worth noting that finite rank operators and  Hilbert-Schmidt operators are absolutely bounded.
Also recall the following definition from \cite{JN}, also see \cite{Sc2,EG}.
\begin{defin}\label{resondef}
Let $Q:=\mathbbm{1}-P$.
We say zero is a regular point of the spectrum
of $H = -\Delta+ V$ provided $ QTQ=Q(U + vG_0v)Q$ is invertible on $QL^2(\mathbb R^2)$.
\end{defin}
(See Definition~\ref{resondef2} in Section~\ref{sec:nonreg} below for a more complete 
description of regularity at zero.)
In Lemma~8 of
\cite{Sc2}, it was proved that if zero is regular, then the operator $D_0:=(QTQ)^{-1}$ is absolutely bounded on $QL^2$.
Below, we discuss the   invertibility of  
$M^\pm(\lambda)=U+vR_0^\pm(\lambda^2)v$, for small 
$\lambda$. This following lemma was proved in Lemma~8 of
\cite{Sc2} by 
employing an abstract Feshbach inversion formula. 

\begin{lemma}\label{Minverse}
    Suppose that zero is a regular point of the spectrum of  $H=-\Delta+V$. Then for   sufficiently small $\lambda_1>0$, the operators
	$M^{\pm}(\lambda)$ are invertible for all $0<\lambda<\lambda_1$ as bounded operators on $L^2(\R^2)$.
	Further, one has
	\begin{align}
	\label{M size}
        	 M^{\pm}(\lambda)^{-1}=h_{\pm}(\lambda)^{-1}S+QD_0Q+ E^{\pm}(\lambda),
	\end{align}
	Here
   	$h_\pm(\lambda)=g^\pm(\lambda)+c$  (with $c\in\R$), and
  	\be\label{S_defn}
  	 	 S=\left[\begin{array}{cc} P & -PTQD_0Q\\ -QD_0QTP & QD_0QTPTQD_0Q
		\end{array}\right]
  	\ee
	is a finite-rank operator with real-valued kernel.  Further, the error term satisfies the bounds
	\begin{align*}
		\big\| \sup_{0<\lambda<\lambda_1} \lambda^{-\frac{1}{2} } |E^{\pm}(\lambda)|\big\|_{HS}
		+\big\| \sup_{0<\lambda<\lambda_1} \lambda^{\frac{1}{2} } |\partial_\lambda E^{\pm}(\lambda)|\big\|_{HS}
		\les 1,
	\end{align*}
	provided that $v(x)\lesssim \langle x\rangle^{-\frac{3}{2}-}$.

\end{lemma}

\noindent
{\bf Remark.}  Under the conditions of Theorem~\ref{thm:main}, the resolvent identity
	\begin{multline}
	    R_V^{\pm}(\lambda^2)=R_0^{\pm}(\lambda^2)-R_0^{\pm}(\lambda^2) v M^{\pm}(\lambda)^{-1}v
	    R_0^{\pm}(\lambda^2) \\ \label{resolvent id}
=R_0^{\pm}(\lambda^2)-R_0^{\pm}(\lambda^2) \frac{v S v}{h_\pm(\lambda)}
	    R_0^{\pm}(\lambda^2)
-R_0^{\pm}(\lambda^2)  v QD_0Q v
	    R_0^{\pm}(\lambda^2) - R_0^{\pm}(\lambda^2)  v E^\pm(\lambda) v
	    R_0^{\pm}(\lambda^2)
	\end{multline}
	holds as an  operator identity  between the spaces $L^{2,\frac{1}{2}+}(\R^2)$ and $ L^{2,-\frac{1}{2}-}(\R^2)$, as in the
	limiting absorption principle, \cite{agmon}.

\subsection{The low energy dispersive bound}

We now consider the evolution 
for low energy, \eqref{spectral rep},  
when $0<\lambda< \lambda_1$.
In contrast with the high energy treatment in 
Section~\ref{sec:high energy} we do
not need any differentiability of initial data 
nor any `regularizing' powers of $H$.  In fact,
for small $\lambda$, we  prove an $L^1\to L^\infty$
dispersive bound.  We use previous work
on dispersive bounds for the two-dimensional Schr\"odinger
equation, using results of \cite{Sc2,EG,EG2} as needed.

First we bound the free resolvent term contribution
to \eqref{resolvent id},

\begin{lemma}\label{r0lowbd}

	We have the 
	bound 
	\begin{align*}
		\sup_{x,y\in \R^2} \bigg|
		\int_0^\infty e^{it\lambda} \chi_1(\lambda) 
		[R_0^+(\lambda)(x,y)-R_0^-(\lambda)(x,y)]
		\, d\lambda\bigg| \les t^{-\f12}.
	\end{align*}
	
\end{lemma}
		
\begin{proof}

	Since $R_0^+(\lambda)(x,y)-R_0^-(\lambda)(x,y)=
	\frac{i}{2}J_0(\lambda|x-y|)$.  We consider the
	contribution of $J_0$ when $\lambda |x-y|\les 1$,
	in which case 
	$J_0(\lambda |x-y|)=1+\widetilde O(\lambda^2|x-y|^2)$,
	see \eqref{J0 def}.
	In this case we note that the boundedness of the 
	integral is clear, to obtain the desired $t^{-1/2}$
	bound we interpolate with
	\begin{align*}
		\bigg|\int_0^\infty e^{it\lambda}\chi_1(\lambda)
		(1+|x-y|^2 \widetilde O(\lambda^2))\, d\lambda
		\bigg| &\les\frac{1}{t}\int_0^\infty 
		\chi_1'(\lambda)+\lambda |x-y|^2 \chi_{\{\lambda\les |x-y|^{-1}\} }\, d\lambda
		\les t^{-1}.
	\end{align*}
	For the piece supported on $\lambda|x-y|\gtrsim 1$,
	we note that for the negative phase piece we have 
	to bound
	\begin{align*}
		\int_0^\infty e^{it\lambda} \chi_1(\lambda)
		e^{-i\lambda |x-y|}\omega_-(\lambda |x-y|)
		\, d\lambda.
	\end{align*}
	As $\omega_-\les 1$, the integral is clearly bounded.
	To attain time decay, we again split into cases.  
	First, if $t-|x-y|\geq \f t2$ (or in the case of
	the `+' phase) we integrate by parts 
	once to bound with
	\begin{align*}
		\frac{1}{t}\int_{\lambda \gtrsim |x-y|^{-1}} \frac{|x-y|^{-\f12}}{\lambda^{\f32}}\, d\lambda
		\les t^{-1}.
	\end{align*}
	On the other hand if $t-|x-y|\leq \f t2$ we note that
	$|x-y|\gtrsim t$ and
	we can bound this contribution by
	\begin{align*}
		\frac{1}{|x-y|^{\f12}}\int_0^1 \lambda^{-\f12}
		\, d\lambda \les t^{-\f12}
	\end{align*}
	as desired.
\end{proof}

We now consider the contribution of the operator
$QD_0Q$ in \eqref{resolvent id}.

\begin{prop}\label{QDQ prop}

	We have the bound
	\begin{multline}\label{QD0Q int}
		\sup_{x,y\in \R^2}
		\bigg| \int_{\R^4}
		\int_{0}^\infty e^{it\lambda} \chi_1(\lambda)
		(1+\lambda)v(x_1)
		QD_0Q(x_1,y_1)v(y_1)\\
		\Big[
		R_0^+(\lambda^2)(x,x_1)R_0^+(\lambda^2)(y_1,y)
		-R_0^-(\lambda^2)(x,x_1)R_0^-(\lambda^2)(y_1,y)
		\Big] \, 
		d\lambda\, dx_1\, dy_1\bigg|
		\les t^{-\f12}.
	\end{multline}

\end{prop}

\begin{proof}
 We can reduce the contribution
to that when there is a single $J_0$ and a single $Y_0$
due to the difference of the `+' and `-' terms in
\eqref{spectral rep} and
\begin{align*}
	&R_0^+(\lambda^2)(x,x_1)R_0^+(\lambda^2)(y,y_1)
	-R_0^-(\lambda^2)(x,x_1)R_0^-(\lambda^2)(y,y_1)\\
	&=-\frac{i}{8}\bigg(Y_0(\lambda|x-x_1|)
	J_0(\lambda|y-y_1|)+J_0(\lambda|x-x_1|)
	Y_0(\lambda|y-y_1|)\bigg).
\end{align*}
This is a beneficial whenever at least one of the
Bessel functions is supported on `low' energy.
For `low-low' interactions, that is when the integral
in \eqref{QD0Q int} has cut-offs $\chi(\lambda|x-x_1|)$
and $\chi(\lambda|y_1-y|)$,
we note that in the two-dimensional Schr\"odinger equation, as in \cite{Sc2}, one wishes to
control integrals of the form
$$
	\int_0^\infty e^{it\lambda^2} \lambda \chi_1(\lambda)
	\mathcal E(\lambda)\, d\lambda
$$
To this end, one integrates by parts once against the
imaginary Gaussian to bound with
$$
	\frac{1}{t}\int_0^\infty\bigg|\frac{d}{d\lambda} 
	\bigg(\chi_1(\lambda)
	\mathcal E(\lambda)\bigg)\bigg|\, d\lambda
$$
Much of the work is involved in finding a uniform bound
for this resulting integral.  In our case, we wish to
bound integrals of the form
$$
	\int_0^\infty e^{it\lambda}(1+\lambda)\chi_1(\lambda)
	\mathcal E(\lambda)\, d\lambda
$$
Integrating by parts once, we are left to bound
$$
	\frac{1}{t}
	\int_0^\infty |	\chi_1(\lambda)
	\mathcal E(\lambda)|+\bigg|\frac{d}{d\lambda} 
	\bigg(\chi_1(\lambda)
	\mathcal E(\lambda)\bigg)\bigg|\, d\lambda
	\les \frac{1}{t}
	\int_0^\infty \bigg|\frac{d}{d\lambda} 
	\bigg(\chi_1(\lambda)
	\mathcal E(\lambda)\bigg)\bigg|\, d\lambda
$$
The last inequality follows from \eqref{J0 def} 
and \eqref{Y0 def}.  Defining
$$
	k(x,x_1):=1+\log^+|x_1|+\log^-|x-x_1|,
$$
with $\log^+ y=\chi_{\{y>1\}}\log y$ and $\log^- y=
-\chi_{\{0<y<1\}}\log y$.
The dispersive bound for Schr\"odinger,
see Lemma~13 in \cite{Sc2}, 
which proved the bound
\begin{multline}\label{Schd2bd}
	\bigg|
	\int_{0}^\infty e^{it\lambda^2}\lambda 
	\chi_1(\lambda)Y_0(\lambda|x-x_1|)
	\chi(\lambda|x-x_1|)v(x_1)
	QD_0Q(x_1,y_1)\\v(y_1)J_0(\lambda |y_1-y|)
	\chi(\lambda|x-x_1|) \, 
	d\lambda\bigg|
	\les \frac{k(x,x_1)}{t},
\end{multline}
implies
the bound for the wave equation.  The desired bound of 
$t^{-\f12}$ follows from a simple interpolation with the
boundedness in time of the integral.  To show this, 
we use the orthogonality relation, $Qv=vQ=0$ to
control the logarithmic singularity of $Y_0$.  
Specifically we note that we can replace
$Y_0(\lambda|x-x_1|)\chi(\lambda |x-x_1|)$ in 
\eqref{QD0Q int} with
\begin{align}
	F(\lambda, x,x_1)=Y_0(\lambda|x-x_1|)\chi(\lambda |x-x_1|)-\frac{2}{\pi}\chi(\lambda(1+|x|))\log(
	\lambda(1+|x|)).\label{F defn}
\end{align}
Noting Lemma~3.3 of \cite{EG} (which arose from the
argument in \cite{Sc2}) we have the bounds on
$F$
\begin{align}\label{F bounds}
	|F(\lambda, x,x_1)|\les k(x,x_1), \qquad
	|\partial_\lambda F(\lambda, x,x_1)|\les 
	\frac{1}{\lambda}.
\end{align}

The boundedness follows from the bounds on $F$
and \eqref{J0 def}.
$$
	\int_0^\infty |\chi_1(\lambda) F(\lambda,x,x_1)
	J_0(\lambda|y-y_1|)
	\chi(\lambda|y-y_1|)|\, d\lambda 
	\les k(x,x_1) \int_0^1 \, d\lambda \les k(x,x_1).
$$
Thus we can conclude that
\begin{multline}
	\bigg|
	\int_{0}^\infty e^{it\lambda}
	\chi_1(\lambda)Y_0(\lambda|x-x_1|)
	\chi(\lambda|x-x_1|)v(x_1)
	QD_0Q(x_1,y_1)\\v(y_1)J_0(\lambda |y_1-y|)
	\chi(\lambda|x-x_1|) \, 
	d\lambda\bigg|
	\les \frac{k(x,x_1)}{t}
\end{multline}
as desired.

Here we'll consider the integrals from the sine operator,
as they are larger in the small $\lambda$ regime.
Consider the `high-low' interaction terms,
$$
	\int_0^\infty e^{it\lambda} \chi_1(\lambda) 
	J_0(\lambda d_1)\widetilde \chi(\lambda d_1)
	Y_0(\lambda d_2)\chi(\lambda d_2)\, d\lambda
	=\int_0^\infty e^{it\lambda} \chi_1(\lambda)
	e^{-i\lambda d_1}\rho_+(\lambda d_1)F(\lambda,y,y_1)
	\, d\lambda.
$$

Consider the case when $t-d_1\geq t/2$, here we can safely
integrate by parts to bound
\begin{align*}
	\frac{1}{t-d_1}\int_0^\infty \bigg|
	\frac{d}{d\lambda}\big(\chi_1(\lambda)
	\rho_+(\lambda d_1)F(\lambda, y,y_1)\big)
	\bigg|\,d\lambda
	&\les \frac{k(y,y_1)}{t}
	\int_0^1 \frac{\lambda^{-\f32}
	\widetilde \chi(\lambda d_1)}{d_1^{\f12}}\, d\lambda.
\end{align*}
We note that the support condition
implies that $\lambda d_1\gtrsim 1$, so we can bound by
\begin{align*}
	\frac{k(y,y_1)d_1^{\f12}}{t} \int_0^1 \lambda^{-\f12}\, d\lambda
	&\les \frac{k(y,y_1)}{t^{\f12}}, 
\end{align*}
where we used that $ d_1\les t$ in the last step.

If $t-d_1\leq \f t2$ then $d_1^{-1}\les t^{-1}$ and
we need to bound an integral of the form
\begin{align*}
	\int_0^1 \frac{F(\lambda, y,y_1)}
	{(1+\lambda d_1)^{\f12}}\, d\lambda
	\les \int_0^1 \frac{k(y,y_1)}{d_1^{\f12}\lambda^{\f12}}
	\, d\lambda \les \frac{k(y,y_1)}{t^{\f12}}.
\end{align*}	
The case where $J_0(\lambda d_1)$ is supported on 
small energy is bounded in the same way as
$|J_0(\lambda d_1)\chi(\lambda d_1)|\les 1\les k(x,x_1)$.
	
We now consider the `high-high' terms.  
Here we do not take advantage of the cancellation between
`+' and `-' resolvents, but instead bound the contribution
of the $R_0^+R_0^+$ and $R_0^-R_0^-$ terms individually.
For the contribution of $R_0^-R_0^-$,
using \eqref{JYasymp2}, we need to
bound an integral of the form
\begin{align*}
	\int_0^\infty e^{it\lambda}
	\chi_1(\lambda)e^{-i\lambda(d_1+d_2)}
	\rho_+(\lambda d_1)\rho_+(\lambda d_2) \, d\lambda
\end{align*}
In the case that $t-(d_1+d_2)\geq \f t2$ we integrate by 
parts to bound integrals of the form
\begin{align*}
	\frac{1}{t}  \int_{\lambda \gtrsim d_1}
	\frac{d_1}{(1+\lambda d_1)^{\f32}(1+\lambda d_2)^{\f12}}
	\, d\lambda
	&\les \frac{1}{t}\frac{1}{d_1^{\f12}d_2^{\f12}}
	\int_{\lambda \gtrsim d_1}
	\frac{1}{\lambda^2}	\, d\lambda
	\les \frac{1}{t}\frac{d_1^{\f12}}{d_2^{\f12}}
	\les \frac{1}{t^{\f12}}\frac{1}{d_2^{\f12}}
\end{align*}	
Where we used that $t\gtrsim d_1,d_2$ in the last line.  
There is, of course, a similar term to consider with 
$d_1$ and $d_2$ switched.
	
On the other hand, if $t-(d_1+d_2)\leq \f t2$ we have that
$t\les \max(d_1, d_2)$ and we need to control an integral
of the form
\begin{align*}
	\int_0^1 \frac{1}{(1+\lambda d_1)^{\f12}
	(1+\lambda d_2)^{\f12}}\, d\lambda\les 
	\frac{1}{t^{\f12}}\int_0^1 
	\lambda^{-\f12}\, d\lambda \les \frac{1}{t^{\f12}}.
\end{align*}
This follows from the bound $t\les \max(d_1, d_2)$
 since
\begin{align*}
	\frac{1}{(1+\lambda d_1)^{\f12}(1+\lambda d_2)^{\f12}}
	&\les \frac{1}{\lambda^{\f12} 
	\max(d_1^{\f12},d_2^{\f12})}\les \frac{1}{t^{\f12}\lambda^{\f12}}.
\end{align*}
We again note that the case of the positive phase,
$R_0^+R_0^+$ follows
implicitly from these arguments with $2t$ replacing
$\f t2$ in the bounds for the two different cases analyzed
when a `high' term is involved in the interaction.
We close the argument by noting that
\begin{align*}
	\sup_{x,y\in \R^2}
	\int_{\R^4} k(x,x_1)v(x_1)&|QD_0Q|(x_1,y_1)
	v(y_1) \frac{1}{|y_1-y|^{\f12}}\, dx_1\, dy_1\\
	&\les \sup_{x,y\in \R^2}
	\|k(x,\cdot) v(\cdot)\|_{L^2}\| |QD_0Q|
	\|_{L^2\to L^2} \big\|v(\cdot)(1+|\cdot-y|^{-\f12})\big\|_{L^2}
	\les 1.
\end{align*}

\end{proof}

We now turn to the contribution of the operator $S$
in \eqref{resolvent id}.

\begin{lemma}\label{lem:S}

	We have the bound
	\begin{multline}
		\bigg|\int_{\R^4}
		 \int_0^\infty e^{it\lambda}\chi_1(\lambda)
		\bigg( \frac{R_0^+(\lambda^2)(x,x_1) v(x_1)
		S(x_1,y_1)v(y_1)
		R_0^+(\lambda^2)(y_1,y)}{h_+(\lambda)}\\
		-\frac{R_0^-(\lambda^2)(x,x_1) v(x_1)
		S(x_1,y_1)v(y_1)
		R_0^-(\lambda^2)(y_1,y)}{h_-(\lambda)}\bigg)\, d\lambda\, dx_1\, dy_1
		\bigg| \les t^{-\f12}
	\end{multline}
	uniformly in $x$ and $y$.

\end{lemma}

\begin{proof}

The `low-low' interaction is bounded using  Lemma~17 of
\cite{Sc2} and the discussion at the start of 
Proposition~\ref{QDQ prop}.  For the other terms we do
not use the difference of `+' and `-' terms but estimate
them individually.  We note the bound
\begin{align}\label{log-bd}
	\chi_1(\lambda)\chi(\lambda|x-x_1|)
	\log(\lambda|x-x_1|)
	\les (1+|\log\lambda|)(1+\log^{-}|x-x_1|).
\end{align}
This bound can be seen by considering the cases of
$|x-x_1|<1$ and $|x-x_1|>1$ separately.
Combining the above bound along with the fact 
that on the 
support of the integrals we have 
$|\partial_\lambda^k h_{\pm}^{-1}(\lambda)|\les \lambda^{-k}|\log \lambda|^{-1}$ 
for $k=0,1$ which allows us to run through the same
argument used for $QD_0Q$.  
As usual we consider the more delicate
case when $R_0^-(\lambda^2)$ is involved.
We first
consider the `high-low'
interaction in which    
we wish to control
\begin{align*}
	\int_0^\infty e^{it\lambda}\chi_1(\lambda)
	e^{-i\lambda d_1} \rho_+(\lambda d_1)
	\frac{\rho_-(\lambda d_2)}{h(\lambda)}\, d\lambda
\end{align*}
In the first case, when $t-d_1\geq \f t2$. 
We first note that the integral above is bounded by
$k(y,y_1)$, since
\begin{align*}
	\int_0^\infty e^{it\lambda}\chi_1(\lambda)
	e^{-i\lambda d_1} \rho_+(\lambda d_1)
	\frac{\rho_-(\lambda d_2)}{h(\lambda)}\, d\lambda
	&\les \int_0^{\lambda_1}
	\frac{(1+|\log\lambda|)(1+\log^{-}|x-x_1|)}
	{|h^-(\lambda)|}\, d\lambda \\
	&\les k(y,y_1)\int_0^{\lambda_1} |\log \lambda|
	+|\log \lambda|^{-1}\, d\lambda
	\les k(y,y_1).
\end{align*}
We can interpolate between this and the
bound obtained by  integrating
by parts against the combined phase
$e^{i\lambda(t-d_1)}$ and use that $(t-d_1)^{-1}\les t^{-1}$.
\begin{align*}
	\frac{1}{t-d_1}\int_0^\infty \bigg| \frac{d}{d\lambda}
	\bigg( \frac{\rho_+(\lambda d_1)
	\rho_-(\lambda d_2)}{h(\lambda)}
	\bigg)\bigg|\, d\lambda
	&\les \frac{k(y,y_1)}{t} \int_{\lambda \gtrsim d_1^{-1}}
	d_1^{-\f12}\lambda ^{-\f32}\, d\lambda \les \frac{k(y,y_1)}{t}.
\end{align*}
In the other case when $t-d_1\leq \f t2$, 
we have $t\les d_1$, we use the decay of 
$\rho_+$ to bound by
\begin{align*}
	\int_0^\infty \bigg|\chi_1(\lambda) 
	\frac{(1+|\log \lambda|)k(y,y_1)}{\lambda^{\f12}d_1^{\f12}}
	\bigg| \, d\lambda
	&\les \frac{k(y,y_1)}{t^{\f12}} \int_0^{\lambda_1}
	\frac{1+|\log\lambda|}{\lambda^{\f12}}\, d\lambda
	\les \frac{k(y,y_1)}{t^{\f12}}.
\end{align*}
The `high-high' interaction is handled similarly by 
considering the size of $t-(d_1+d_2)$ compared to
$\frac{t}{2}$ as in the `high-high' interactions
considered in Proposition~\ref{QDQ prop}.  As in the
previous proofs, when $R_0^+(\lambda^2)$ is involved
one needs only do the case analysis by comparing
the size  $t+d_1$ or
$t+d_1+d_2$ with $2t$.

\end{proof}

Finally we turn to the error term, $E^{\pm}(\lambda)$
in \eqref{resolvent id}.  In this case, we note that 
all the functions of $\lambda$ and
their derivatives are smaller than those encountered
in the $QD_0Q$ and/or $S$ terms. 

\begin{lemma}\label{lem:Ebd}

	We have the bound
	\begin{align*}
		\sup_{x,y\in \R^2}\bigg| \int_{\R^4}
		\int_0^\infty e^{it\lambda} \chi_1(\lambda)
		R_0^{\pm}(\lambda^2)(x,x_1)v(x_1)E^{\pm}(\lambda)
		(x_1,y_1)v(y_1) 
		R_0^{\pm}(\lambda^2)(y_1,y) \, d\lambda
		\, dx_1\, dy_1 \bigg|\les
		t^{-\f12}.
	\end{align*}

\end{lemma}

\begin{proof}
We provide a brief outline of the proof 
the convience of
the reader.  We again reduce the analysis to a series
of cases, `low-low', 'high-low' and 'high-high' as
in the proof of Proposition~\ref{QDQ prop}.  Here,
we use the vanishing of the error term at $\lambda=0$
of the error term, see Lemma~\ref{Minverse}, allows
us to control the logarithmic singularities of the
resolvent.  We cannot utilize the orthogonality relations
we used for $QD_0Q$ or the difference of `+' and `-'
terms, though the smallness of $E^{\pm}(\lambda)$ 
more than compensates for this loss.

We provide a sketch of how one controls the `high-low'
interaction for the convience of the reader.  
That is, we wish to control integrals of the form
\begin{align*}
	\int_0^\infty e^{it\lambda} \chi_1(\lambda)
	e^{-i\lambda d_1}\rho_+(\lambda d_1)
	E^{\pm}(\lambda) \rho_-(\lambda d_2) \, d\lambda.
\end{align*}
Using \eqref{log-bd} it is easy to see the above
integral is bounded by $(1+\log^- d_2)$.  
That is we bound with
\begin{align*}
	\sup_{0<\lambda<\lambda_1} \big|\lambda^{-\f12}
	E^{\pm}(\lambda)\big|
	\int_0^{\lambda_1} |\lambda^{\f12} \rho_+(\lambda d_1)
	\rho_-(\lambda d_2)|\, d\lambda
	&\les (1+\log^- d_2) \sup_{0<\lambda<\lambda_1} \big|\lambda^{-\f12} E^{\pm}(\lambda)\big|.
\end{align*}
By Lemma~\ref{Minverse}, the error term defines a
bounded operator on $L^2$.

As usual, for the time decay
we have two cases to consider.  
If $t-d_1\geq \f t2$ we integrate by
parts and note that
\begin{multline*}
	\bigg|\frac{d}{d\lambda}\bigg(
	\chi_1(\lambda)
	\rho_+(\lambda d_1) \rho_-(\lambda d_2)
	\bigg)\bigg|\\
	\les \frac{k(x,x_1)(1+|\log \lambda|)}{\lambda^{\f12}}
	\bigg(
	\sup_{0<\lambda<\lambda_1} \big|\lambda^{-\f12}
	E^{\pm}(\lambda)\big|+
	\sup_{0<\lambda<\lambda_1} \big|\lambda^{\f12}
	\partial_\lambda E^{\pm}(\lambda)\big|\bigg).
\end{multline*}
By Lemma~\ref{Minverse} the operators involving
the error term are absolutely bounded.  The $t^{-1}$
bound follows from the observations that
$\lambda^{-\f12}(1+|\log\lambda|)$ is integrable on
the support of $\chi_1$ and
$k(x,x_1)v(x_1)\in L^2_{x_1}$ uniformly in $x$.

The case when $t-d_1\leq \f t2$ is simpler.  Here we use
$t\les d_1$ to bound with
\begin{align*}
	\frac{k(x,x_1)}{t^{\f12}}
	\int_0^\infty \chi_1(\lambda) | \lambda^{-\f12}
	E^{\pm}(\lambda)| (1+|\log\lambda|)\, d\lambda
\end{align*}
The bound follows from Lemma~\ref{Minverse} as in the previously considered case.
The `high-high' interaction is handled similarly by 
considering the size of $t-(d_1+d_2)$.

We note that the `low-low' interaction is handled by
Lemma~18 in \cite{Sc2} along with the discussion in
the proof of Proposition~\ref{QDQ prop}.  This
$t^{-1}$ bound is interpolated with the clear boundedness,
due to the integrability of $\lambda^{\f12}(\log \lambda)^2$ on the support of $\chi_1$ by treating
the error term as in the mixed `high-low' case above.

\end{proof}

\begin{proof}[Proof of Proposition~\ref{prop:lowbd}]

	The proposition follows from the expansion
	\eqref{resolvent id},
	Lemma~\ref{r0lowbd},
	Proposition~\ref{QDQ prop},
	Lemmas~\ref{lem:S}, \ref{lem:Ebd}.

\end{proof}

\section{Weighted Estimates}\label{sec:weighted}

In this section, we prove a weighted version of the
dispersive decay.  At the cost of polynomially growing
spatial weights,
we can gain decay in $t$ to acheive a decay rate of
$t^{-1}(\log t)^{-2}$ for large $t$, which is integrable
as $t\to \infty$.  Such integrable dispersive bounds have
applications in the analysis of non-linear equations,
see \cite{BP,PW,SW,Wed} for example.
This result is obtained in a manner
similar to the weighted decay for the two-dimensional
Schr\"odinger equation obtained in \cite{Mur,EG},
and is motivated by recent work of Kopylova 
in \cite{Kop}.

Roughly speaking, we are taking advantage of the fact that the free Schr\"odinger operator $H_0=-\Delta$ has
a resonance at zero energy in dimension  two.  
It is well-known that 
a resonance at zero energy slows the long-time decay
rate of the evolution operator, see 
\cite{JenKat,ES,KS,EG}
and Section~\ref{sec:nonreg} below.  By perturbing
with a potential that removes this zero energy resonance,
one can obtain a faster time decay rate by allowing for
spatially growing weights.  A similar gain of time
decay rate can be seen in one spatial dimension,
\cite{Scs}.

The condition of zero being
regular is equivalent to the boundedness of the operators
$R_V^{\pm}(\lambda^2)$ between certain weighted 
$L^2$ spaces.  We can see from the expansion in
Lemma~\ref{R0 exp cor} that zero is not regular for the
free resolvent due to the singular $\log\lambda$ term.

Due to the representation \eqref{spectral rep}, 
Theorem~\ref{thm:wtd} follows from the following
oscillatory integral bound.
\begin{prop}\label{prop:wtdmain}

	If $|V(x)|\les \la x\ra^{-3-\alpha}$ for some
	$\alpha\in (0,\frac{1}{4})$ and zero is a regular
	point of the spectrum of $H=-\Delta+V$, then
	for $t>2$,
	\begin{align*}
		\bigg|\int_0^\infty \big(\sin(t\lambda)
		\la \lambda \ra^{-\f14-}+\lambda
		\cos(t\lambda) \la \lambda \ra^{-\f34-}\big)
		[R_V^+(\lambda^2)(x,y)-R_V^-(\lambda^2)(x,y)]
		\,d\lambda \bigg| \les 
		\frac{\la x\ra^{\f12+\alpha+} 
		\la y\ra^{\f12+\alpha+}}{t (\log t)^2}.
	\end{align*}

\end{prop}

The proof of this bound follows in a similar manner as
the non-weighted $t^{-\f12}$ bound.  Accordingly,
we divide this section into two subsections.  We first 
control the high energy, when $\lambda>\lambda_1>0$ 
with a faster time decay
when allowing for weights, then
we state revised resolvent expansions from \cite{EG2}
for low energy that allow us to prove the desired time
decay for low energy.

\subsection{High energy}

Much of the care we took in Section~\ref{sec:high energy}
was to avoid growth in the spatial variables by avoiding
differentiating the phase $e^{i\lambda |x-y|}$.  If
we allow for a spatially weighted estimate, one can proceed in
a less delicate manner.  

As usual, the high energy portion of 
Proposition~\ref{prop:wtdmain} follows from

	\begin{prop}\label{prop:wtdhi}
	
		If $|V(x)|\les \la x\ra^{-3-\alpha}$ for some
		$\alpha\in (0,\frac{1}{4})$ and zero is a regular
		point of the spectrum of $H=-\Delta+V$, then
		for $t>2$,
		\begin{align*}
			\bigg|\int_0^\infty e^{it\lambda} 
			\widetilde \chi_1(\lambda)\lambda ^{-\f12-} 
			[R_V^+(\lambda^2)(x,y)-R_V^-(\lambda^2)(x,y)]
			\,d\lambda \bigg| \les 
			\frac{\la x\ra^{\f12+\alpha} 
			\la y\ra^{\f12+\alpha}}{t^{1+\alpha}}.
		\end{align*}
	
	\end{prop}

The high energy bounds follow with some modifications
to Lemma~\ref{lem:bs2hi} and 
\ref{lem:tailhi} which are outlined below.
We again start with the resolvent
expansion \eqref{BS:high}. 
We 
consider the contribution of the free resolvent
with the difference of the `+' and `-' terms so that for
the first term in the Born series, \eqref{BS:high}, we
need only consider the contribution of the Bessel function
$J_0$.

\begin{lemma}\label{lem:J0hiwtd}

	For $t>2$, we have the bound 
	$$		
		\sup_{L\geq 1}
		\bigg|\int_0^\infty e^{it\lambda} \widetilde \chi_1(\lambda) \chi_L(\lambda)
		\lambda^{-\f12-} J_0(\lambda |x-y|)\, d\lambda
		\bigg| \les \frac{\la x\ra^{\f12}\la y \ra^{\f12}}{t^{\f32}}+ \frac{1}{t^2}.
	$$

\end{lemma}

\begin{proof}

We first consider the
$J_0$ on $\lambda|x-y|\les 1$.  By the expansion
\eqref{J0 def}, we need only  bound
\begin{align*}
	\int_0^\infty e^{it\lambda} \widetilde \chi_1(\lambda)
	\chi_L(\lambda)
	\lambda^{-\f12-} \widetilde O(1)\, d\lambda
\end{align*}
We note that $\chi^\prime$ is supported on 
$\lambda \approx 1$ and we have
\begin{align*}
	|\partial_\lambda^k \widetilde \chi_1(\lambda)|,
	|\partial_\lambda^k \chi_L(\lambda)| \les \lambda^{-k}.
\end{align*}
Integrating by parts twice against the phase
$e^{it\lambda}$ yields the bound of
\begin{align*}
	\frac{1}{t^2}
	\int_0^\infty \bigg| \frac{d}{d\lambda} \bigg(
	\widetilde \chi_1(\lambda) \chi_L(\lambda)
	\lambda^{-\f12-} \widetilde O(1)\bigg)\bigg|
	\, d\lambda
	&\les \frac{1}{t^2}\int_1^\infty \lambda^{-\f52-}
	\, d\lambda \les t^{-2}.
\end{align*}
We next consider when $\lambda|x-y|\gtrsim 1$ 
by \eqref{JYasymp2} we need to bound
bound
\begin{align*}
	\int_0^\infty e^{it\lambda}\widetilde \chi_1(\lambda)
	\chi_L(\lambda)
	\lambda^{-\f12-}e^{\pm i \lambda |x-y|}\omega_{\pm}
	(\lambda |x-y|)\, d\lambda
\end{align*}
Recall that $\omega_{\pm}(z)=\widetilde{O}(z^{-\f12})$
and is supported on $z\gtrsim 1$.
We consider the case of the `+' phase along with
the case of the `-' phase when $t-|x-y|\geq \f t2$
together.  In both cases, we can integrate by parts
twice against the phase $e^{i\lambda(t\pm |x-y|)}$ safely
and bound with
\begin{align*}
	\frac{1}{t^2}\int_0^\infty \bigg| 
	\frac{d^2}{d\lambda^2}\bigg( 
	\widetilde \chi_1(\lambda) \chi_L(\lambda)
	\lambda^{-\f12-}\omega_{\pm}(\lambda|x-y|)\bigg)
	\bigg|\, d\lambda
	&\les t^{-2} \int_1^\infty \lambda^{-\f52-}\, d\lambda
	\les t^{-2}.
\end{align*}
For the `-' phase, we consider the second case in which
$t-|x-y|\leq \f t2$, in which case $t\les |x-y|$.  Here
we integrate by parts once against $e^{it\lambda}$ and
bound with
\begin{align*}
	\frac{1}{t}\int_0^\infty \bigg| \frac{d}{d\lambda}
	\bigg(\widetilde \chi_1(\lambda) \chi_L(\lambda)
	\lambda^{-\f12-}
	e^{-i\lambda |x-y|}\omega_-(\lambda|x-y|)
		\bigg)\bigg|\, d\lambda
	&\les \frac{1+|x-y|}{t|x-y|^{\f12}}
	\int_1^\infty \lambda^{-1-}\, d\lambda\\
	&\les \frac{1+|x-y|^{\f12}}{t^{\f32}}\les
	\frac{\la x\ra^{\f12} \la y\ra^{\f12}}{t^{\f32}}.
\end{align*}
Here we see why polynomial weights are needed to 
gain the extra time decay.

\end{proof}

For the remaining terms in \eqref{BS:high} we note that
Lemmas~\ref{lem:bs2hi} and \ref{lem:tailhi} ensure the
convergence of the $\lambda$ integrals we wish to
analyze.  Accordingly, we can omit the $\chi_L$
cut-off function.
We now turn to the second term of the Born series
\eqref{BS:high} and estimate the `+' and `-' terms
separately.  

\begin{lemma}\label{lem:bs2hiwtd}

	If $|V(x)|\les \la x\ra^{-3-\alpha}$, then for $t>2$ we 
	have the bound
	\begin{align*}
		\bigg| \int_{\R^2}
		\int_0^\infty e^{it\lambda}\lambda^{-\f12-} 
		\widetilde{\chi}_1(\lambda) 
		R_0^{\pm}(\lambda^2)(x,z)V(z)
		R_0^{\pm}(\lambda^2)(z,y)\,d\lambda\, dz
		\bigg| \les \frac{\la x\ra^{\f12+\alpha}
		\la y\ra^{\f12+\alpha}}{t^{1+\alpha}}.
	\end{align*}

\end{lemma}

\begin{proof}

This proof follows along the lines of the proof of
Lemma~\ref{lem:bs2hi}.  We recall the 
expansion for $R_0^{\pm})(\lambda^2)(x,y)$ given 
in \eqref{Resolvent decomp}.
For the `low-low' interaction, 
integrating by parts twice against the phase
$e^{it\lambda}$ yields
\begin{multline*}
	\bigg| \int_0^\infty e^{it\lambda}\widetilde \chi_1(\lambda) \lambda^{-\f12-} \rho_-(\lambda d_1)
	\rho_-(\lambda d_2)\, d\lambda \bigg|
	\les \frac{1}{t^2}\int_0^\infty \bigg| \frac{d^2}
	{d\lambda^2} 
	(\widetilde \chi_1(\lambda)\lambda^{-\f12-}
	\rho_-(\lambda d_1) \rho_-(\lambda d_2))\bigg| \, 
	d\lambda\\
	\les \frac{(1+\log^- d_1)(1+\log^- d_2)}{t^2}
	\int_1^\infty 
	\lambda^{-\f52-}\, d\lambda
	\les \frac{(1+\log^- d_1)(1+\log^- d_2)}{t^2}.
\end{multline*}
Here we used the inequality \eqref{log-eqn} to bound
with the $\log^- d_j$ spatial weights.
For the `high-low' and `high-high' interactions, we 
consider only the negative phase contributions.  
For the `high-low' interaction, we again consider cases.
If $t-d_1\geq \f t2$, then we can integrate by parts
against the combined phase $e^{i\lambda(t-d_1)}$,
\begin{multline*}
	\bigg| \int_0^\infty e^{i\lambda(t-d_1)} \widetilde 
	\chi_1(\lambda) \lambda^{-\f12-} \rho_-(\lambda d_2)
	\rho_+(\lambda d_1)\, d\lambda \bigg|\\
	\les \frac{1}{(t-d_1)^2}\int_0^\infty \bigg|
	\frac{d^2}{d\lambda^2}(\widetilde \chi_1(\lambda) 
	\lambda^{-\f12-} \rho_-(\lambda d_2)
	\rho_+(\lambda d_1))\bigg|\, d\lambda\\
	\les \frac{1+\log^- d_2}{t^2 d_1^{\f12}} 
	\int_1^\infty \lambda^{-\f52-}\, d\lambda\les 
	\frac{1+\log^- d_2}	{t^2 d_1^{\f12}}.
\end{multline*}
On the other hand, if $t-d_1\leq \f t2$ then we integrate
by parts once against the phase $e^{it\lambda}$
and, similar to the case of the free
resolvent only, use the decay of the Bessel function
and that $t\les d_1$ to see
\begin{multline*}
	\bigg| \int_0^\infty e^{it\lambda} \widetilde 
	\chi_1(\lambda) \lambda^{-\f12-} \rho_-(\lambda d_2)
	e^{-i\lambda d_1}\rho_+(\lambda d_1)\, 
	d\lambda \bigg|\\
	\les \frac{1}{t}\int_0^\infty \bigg|
	\frac{d}{d\lambda}(\widetilde \chi_1(\lambda) 
	\lambda^{-\f12-} \rho_-(\lambda d_2)e^{-i\lambda d_1}
	\rho_+(\lambda d_1))\bigg|\, d\lambda\\
	\les \frac{(1+\log^- d_2)\la d_1\ra}{t d_1^{\f12}} \int_1^\infty 
	\lambda^{-1-}\, d\lambda\les 
	\frac{(1+\log^- d_2)\la x\ra^{\f12} \la z\ra^{\f12}}
	{t^{\f32}}.
\end{multline*}
There is a corresponding term with $d_1$ and $d_2$ 
interchanged.
The `high-high' interaction is handled in cases as well.
If $t-(d_1+d_1)\geq \f t2$ we use that
\begin{multline*}
	\bigg| \int_0^\infty e^{i\lambda(t-(d_1+d_2))} 
	\widetilde \chi_1(\lambda) \lambda^{-\f12-} 
	\rho_+(\lambda d_2)
	\rho_+(\lambda d_1)\, d\lambda \bigg|\\
	\les \frac{1}{t^2}\int_0^\infty \bigg|
	\frac{d^2}{d\lambda^2}(\widetilde \chi_1(\lambda) 
	\lambda^{-\f12-} \rho_+(\lambda d_2)
	\rho_+(\lambda d_1))\bigg|\, d\lambda
	\les \frac{1}{t^2 d_1^{\f12}d_2^{\f12}} 
	\int_1^\infty \lambda^{-\f52-}\, d\lambda\les 
	\frac{1}{t^2 d_1^{\f12}d_2^{\f12}}.
\end{multline*}
On the other hand, if $t-(d_1+d_2)\leq \f t2$, then
$t\les \max (d_1,d_2)$ and we see that
\begin{multline*}
	\bigg| \int_0^\infty e^{it\lambda} \widetilde 
	\chi_1(\lambda) \lambda^{-\f12-} 
	e^{-i\lambda (d_1+d_2)}\rho_+(\lambda d_2)
	\rho_+(\lambda d_1)\, 
	d\lambda \bigg|\\
	\les \frac{1}{t}\int_0^\infty \bigg|
	\frac{d}{d\lambda}(\widetilde \chi_1(\lambda) 
	\lambda^{-\f12-} e^{-i\lambda (d_1+d_2)}
	\rho_+(\lambda d_2)
	\rho_+(\lambda d_1))\bigg|\, d\lambda\\
	\les \frac{1+d_1+d_2}{t d_1^{\f12}d_2^{\f12}} 
	\int_1^\infty \lambda^{-1-}\, d\lambda\les 
	\frac{1+\max(d_1,d_2)}{t \min(d_1^{\f12}, d_2^{\f12})\max(d_1^{\f12}, d_2^{\f12})}
	\les \frac{(\la x\ra \la z \ra \la y 
	\ra)^{\f12+\alpha}}
	{t^{1+\alpha} \min(d_1^{\f12}, d_2^{\f12})}.
\end{multline*}
To conclude the desired bounds, we note that
under the assumptions of Proposition~\ref{prop:wtdmain}
\begin{multline*}
	\int_{\R^2} |V(z)| \bigg( (1+\log^- |x-z|)(1+|z-y|^{-\f12})+\la z\ra^{\f12+\alpha} |x-z|^{-\f12}
	+(|x-z|^{-\f12}|z-y|^{-\f12})
	\bigg)\, dz \les 1
\end{multline*}
uniformly in $x$ and $y$.

\end{proof}

\begin{proof}[Proof of Proposition~\ref{prop:wtdhi}]

The Proposition follows from Lemma~\ref{lem:J0hiwtd},
Lemma~\ref{lem:bs2hiwtd} and the following modification
of Lemma~\ref{lem:tailhi} in 
Section~\ref{sec:high energy}.
For the tail of the Born series, we note that we can
essentially repeat the argument of Lemma~\ref{lem:tailhi}
but integrate by parts twice in \eqref{I def} to
see
\begin{align*}
	|\eqref{I def}|&\les \frac{1}{t^2}\int_0^\infty
	\bigg|	\frac{d^2}{d\lambda^2} \bigg( 
	\lambda^{-\f12-} \mathcal E(\lambda)(x,y) \bigg)
	\, d\lambda \bigg|\les \frac{\la x\ra^{\f12} 
	\la y\ra^{\f12}}{t^2}.
\end{align*}
Here we use the bound \eqref{R0 wtdL2} directly instead
of separating the phase from the large $\lambda r$ 
portion.  This, of course, leaves us with 
polynomial weights and requires an extra power of decay
on the potential.  This is easily seen since differentiating the resolvents twice requires a potential
that maps $L^{2,-\f52-}$ to $L^{2,\f12+}$.

\end{proof}

\subsection{Low Energy}
To attain the extra time decay, we need to modify 
slightly the resolvent expansions laid out in
Section~\ref{sec:low energy}.  Accordingly, we
use the resolvent expansions developed for the analysis
of the Schr\"odinger equation from \cite{EG}.

\begin{prop}\label{lowprop} Fix $0<\alpha<\f14$. Let $v(x)\lesssim \la x\ra^{-\f32-\alpha-}.$
For any $t > 2$, we have
\begin{multline}\label{stone2}
	\Big|\int_0^\infty (\sin(t\lambda)
	+\lambda \cos(t\lambda))
	\chi_1(\lambda) [R_V^+(\lambda^2)
	-R_V^-(\lambda^2)](x,y)
	d\lambda\Big| \\
	\les  \f{(1+\log^+ |x|)(1+\log^+ |y|)}{t\log^2(t)} + 
	\f{\la x\ra^{\f12+\alpha+} \la 
	y\ra^{\f12+\alpha+}}{t^{1+\alpha} }.
\end{multline}
\end{prop}

The following corollary of Lemma~\ref{R0 exp cor}	 
follows from the bounds for $\partial_\lambda E_0^\pm$ 
and $\partial^2_\lambda E_0^\pm$.
\begin{corollary} \label{lipbound} For $0<\alpha<1$  and $b>a>0$ we have
$$
|\partial_\lambda E_0^\pm(b)-\partial_\lambda E_0^\pm(a)|\les a^{-\f12} |b-a|^{\alpha} |x-y|^{\frac12+\alpha}.
$$
\end{corollary}

\begin{proof}

	The bounds on $\partial_\lambda^2 E_0^{\pm}$
	from Lemma~\ref{R0 exp cor} along with the mean
	value theorem guarantee that
	$$
	|\partial_\lambda E_0^\pm(b)-\partial_\lambda E_0^\pm(a)|\les a^{-\f12}|b-a| |x-y|^{\f32}.	
	$$
	The statement of this corollary then follows from
	a simple interpolation with the bounds in
	Lemma~\ref{R0 exp cor}.

\end{proof}

This corollary implies the improved error bounds
below.

\begin{lemma} \label{lem:M_exp_E} Let $0<\alpha<1$.
	The error term in Lemma~\ref{lem:M_exp}  
	satisfies the bound
	\begin{multline*}
		\big\| \sup_{0<\lambda<\lambda_1} \lambda^{-\frac{1}{2}} |E_1^{\pm}(\lambda)|\big\|_{HS}
		+\big\| \sup_{0<\lambda<\lambda_1} \lambda^{\frac{1}{2}} |\partial_\lambda E_1^{\pm}(\lambda)|\big\|_{HS}	
		\\+\big\| \sup_{0<\lambda<b<\lambda_1} \lambda^{\frac{1}{2}} (b-\lambda)^{-\alpha} |\partial_\lambda E_1^{\pm}(b)-\partial_\lambda E_1^\pm(\lambda)|\big\|_{HS}	
		\les 1,
	\end{multline*}
	provided that $v(x)\lesssim \langle x\rangle^{-\frac{3}{2}-\alpha-}$.
	
\end{lemma}

Recall that from Lemma~\ref{Minverse} we have the
expansion
	\begin{align*}
	 M^{\pm}(\lambda)^{-1}=h_{\pm}(\lambda)^{-1}S
	 +QD_0Q+ E^{\pm}(\lambda),
	\end{align*}

\begin{lemma}\label{Minverse2}
    Let $0<\alpha<1$.
    Under the conditions of Lemma~\ref{Minverse}, 
    we have the improved error bounds
	\begin{multline*}
		\big\| \sup_{0<\lambda<\lambda_1} \lambda^{-\frac{1}{2} } |E^{\pm}(\lambda)|\big\|_{HS}
		+\big\| \sup_{0<\lambda<\lambda_1} \lambda^{\frac{1}{2} } |\partial_\lambda E^{\pm}(\lambda)|\big\|_{HS}	\\
		+\big\| \sup_{0<\lambda<b\les\lambda<\lambda_1}  \lambda^{\frac{1}{2}+\alpha} (b-\lambda)^{-\alpha} |\partial_\lambda E^{\pm}(b)-\partial_\lambda E^\pm(a)| \big\|_{HS}	
		\les 1,
	\end{multline*}
	provided that $v(x)\lesssim \langle x\rangle^{-\frac{3}{2}-\alpha-}$.

\end{lemma}

We note that for fixed $x,y$ the kernel 
$R_V^\pm(\lambda^2)(x,y)$ of the  resolvent  remains 
bounded as $\lambda \to 0$. This is because of a 
cancellation between the first and second summands of the 
second line of the expansion for $R_V^{\pm}$ 
in \eqref{resolvent id}. A consequence of 
this cancellation is cancellation of the slowest
decaying terms in the evolution which 
makes it possible to obtain the faster
time decay.

We first establish the following simple Lemma 
which allows us to use the Lipschitz
bounds on the error terms from Lemma~\ref{Minverse2}.
This Lemma is crucial on two fronts, it allows us to
minimize both the decay assumptions on the potential and
the size of polynomial weights needed in 
Theorem~\ref{thm:wtd}, see Proposition~\ref{freeevol},
Lemma~\ref{lem:Sll3}, Proposition~\ref{prop:stone2_3},
and Proposition~\ref{prop:stone2er} below.
\begin{lemma}\label{lem:ibp2} Assume that  $\mathcal E(0)=0$. For $t>2$, we have
\begin{align}\label{ibp2}
	\Big|\int_0^\infty e^{\pm it\lambda} \, \mathcal E(\lambda) d\lambda  \Big| 
	\les \f1t\int_0^{\infty}\frac{|\mathcal E^\prime(\lambda)|}{  (1+\lambda t)} d\lambda
	+\f1{t}\int_{t^{-1}}^\infty \big|  \mathcal E^\prime(\lambda +\pi/t)-\mathcal E^\prime(\lambda)  \big| d\lambda.
\end{align}
\end{lemma}
\begin{proof}
We integrate by parts
once and then consider two different regions.  We give the
proof for the positive phase, the negative phase follows
identically.
$$
 	\int_0^\infty e^{it\lambda} \mathcal E(\lambda) 
 	d\lambda =  -\f{i}{t}  \int_0^\infty e^{it\lambda}  
 	\mathcal E^\prime(\lambda) d\lambda
$$
We divide this integral into two pieces.  The first region
$0<\lambda<\frac{2\pi}{t}$, can clearly be bounded by the first
integral in the statement of the Lemma.
On the second region, $\frac{2\pi}{t}<\lambda$, we note that
$$
	\int_{\f{2\pi}{t}}^\infty e^{it\lambda}  
	\mathcal E^\prime(\lambda)\, d\lambda 
	=-\int_{\f{2\pi}{t}}^\infty e^{it(\lambda-\f\pi{t})}  
	\mathcal E^\prime(\lambda)\, d\lambda 
	= -\int_{\f{\pi}{t}}^\infty e^{it\lambda}  
	\mathcal E^\prime(\lambda+\pi/t)\, d\lambda.
$$
Therefore it suffices to consider (the integral on $[\pi/t,2\pi/t]$ is bounded by the first integral on the right hand side of \eqref{ibp2})
$$
	\int_{\f{\pi}{t}}^\infty e^{it\lambda}  
	\Big(\mathcal E^\prime(\lambda) 
	-\mathcal E^\prime(\lambda+\pi/t)\Big) \, d\lambda.
$$

\end{proof}

We start with the contribution of the free resolvent to \eqref{stone2}. In this instance we work with the
functions $\sin(t\lambda)$ and $\lambda \cos(t\lambda)$
directly rather than estimating terms containing
$e^{\pm i\lambda t}$.  This will allow us to 
explicitly determine the slowest decaying 
(order $t^{-1}$) 
term.
\begin{prop}\label{freeevol} For any $0<\alpha<\f14$,
we have
$$\int_0^\infty (\sin(t\lambda)+\lambda \cos(t\lambda)) \chi_1(\lambda)  [R_0^+(\lambda^2)-R_0^-(\lambda^2)](x,y)
d\lambda=\frac{i}{2t} +O\Big(\f{\la x \ra^{\f12 +\alpha } \la y\ra^{\f12 +\alpha }}{t^{1+\alpha}}\Big).
$$
\end{prop}
\begin{proof}
Using Lemma~\ref{R0 exp cor}, we have
$$
R_0^+-R_0^-=\f{i}{2}+E_0^+(\lambda)-E_0^-(\lambda).
$$
We first consider the $\sin(t\lambda)$ term, the 
$\lambda \cos(t\lambda)$ term is smaller.
We now rewrite the $\lambda$ integral above as
\begin{align*}
\frac{i}{2}\int_0^\infty \sin(t\lambda) \chi_1(\lambda)
d\lambda + \int_0^\infty \sin(t\lambda) \chi_1(\lambda) (E_0^+(\lambda)-E_0^-(\lambda)) d\lambda=:I+II.
\end{align*}
Note that by integrating by parts we obtain
\begin{align}\label{Aest}
I=\frac{i}{2}\bigg(\frac{-\cos(t\lambda)\chi_1(\lambda)}
{t}\bigg|^\infty_0+\frac{1}{t}\int_0^\infty \cos(t\lambda)
\chi_1'(\lambda)\, d\lambda
\bigg)=\frac{i}{2t}+O(t^{-2})
\end{align}
Since
$\chi_1'$ is supported on $\lambda \approx 1$, 
the second term can be integrated by parts again.

Using the bounds in Lemma~\ref{R0 exp cor} for 
$\mathcal E(\lambda)= \chi_1(\lambda) 
(E_0^+(\lambda)-E_0^-(\lambda)) $, we see that 
$\mathcal E(0)=0$, and from
Corollary~\ref{lipbound},
\begin{align*}
|\partial_\lambda \mathcal E(\lambda)| &\les   \lambda^{-\f12}|x-y|^{\f12 } \les \lambda^{-\f12 } \sqrt{\la x \ra \la y\ra},\\
 \Big|\partial_\lambda \mathcal E (\lambda+\pi/t)
 -\partial_\lambda \mathcal E(\lambda)\Big|& \les \lambda^{-\f12-\alpha}t^{-\alpha}\la x-y\ra^{\f12+\alpha}
 \end{align*}
Accordingly, we decompose sine and cosine into
terms involving $e^{\pm i\lambda t}$ use Lemma~\ref{lem:ibp2} to obtain
$$
	|II|\les  \f1t\int_0^{\infty}\frac{|\mathcal E^\prime(\lambda)|}{  (1+\lambda t)} d\lambda
	+\f1{t}\int_{t^{-1/2}}^\infty \big|  \mathcal E^\prime(\lambda +\pi/t)-\mathcal E^\prime(\lambda)\big| d\lambda.
$$
Using the bounds above, we estimate the first integral by
$$
 \f{ \sqrt{ \la x\ra \la y\ra } }{t}\int_0^{\infty}\frac{1}{ \sqrt{\lambda} (1+\lambda t)} d\lambda\les  \f{ \sqrt{ \la x\ra \la y\ra } }{t^{1+\alpha}}.
$$
To estimate the second integral, we apply the Lipschitz bound to see
$$
	\f{ (\la x\ra \la y
	\ra)^{\f12+\alpha}}{t}\int_{t^{-1}}^{\lambda_1}  
	\lambda^{-\f12-\alpha} (t\lambda)^{-\alpha} d\lambda
	\les \f{ (\la x\ra \la 
	y\ra)^{\f12+\alpha }  }{t^{1+\alpha}},
$$
since $\alpha \in (0,\f14)$.

For the cosine term we note that we need to control
and integral of the form 
\begin{align*}
	\int_0^\infty \cos(t\lambda)  
	\mathcal E_1(\lambda)\, d\lambda,
\end{align*}
with $\mathcal E_1(\lambda)=\lambda\chi_1(\lambda)
(\frac{i}{2}+E_0^+(\lambda)-E_0^-(\lambda))$.  Here, we
can integrate by parts  without an order
$t^{-1}$ boundary term.
This follows from the vanishing of $\mathcal E_1(\lambda)$
and at $\lambda=0$.  With
the bounds on $E_0^{\pm}$ we can  bound this 
contribution as we did term II.

\end{proof}

We now consider the contribution of the second term in \eqref{resolvent id} to \eqref{stone2}:
\be\label{stone2_2}
 \int_{\R^4}\int_0^\infty 
 (\sin(t\lambda)+\cos(t\lambda)\lambda )
 \chi_1(\lambda)  [\mathcal R^- -\mathcal R^+] v(x_1)S(x_1,y_1)v(y_1)
d\lambda dx_1 dy_1,
\ee
where
\be\label{calR}
\mathcal R^\pm=\f{R_0^{\pm}(\lambda^2)(x,x_1) R_0^{\pm}(\lambda^2)(y_1,y)}{h_\pm(\lambda)}.
\ee

\begin{prop}\label{prop:stone2_2} Let $0<\alpha<1/4$. If $v(x)\les \la x\ra^{-\f32-\alpha-}$, then
we have
$$
\eqref{stone2_2}=-\frac{i}{2t} + O\Big( \f{(1+\log^+ |x|)(1+\log^+|y|)  }{t(\log(t))^2} \Big)+O\Big( \f{\la x\ra^{\f12+\alpha+} \la y\ra^{\f12+\alpha+}}{t^{1+\alpha} } \Big).
$$
\end{prop}
We will prove this Proposition through a series of
Lemmas to control the various terms that arise from the
difference of the `+' and `-' terms.
Recall from  Lemma~\ref{R0 exp cor}  that
	\begin{align*}
		R_0^{\pm}(\lambda^2)(x,x_1)=\frac{1}{\|V\|_1} g^{\pm}(\lambda)+G_0(x,x_1)+E_0^\pm(\lambda)(x,x_1).
	\end{align*}
Recalling that $h^\pm(\lambda)=g^\pm(\lambda)+c$ 
with $c\in \R$, we see
$$
\mathcal R^\pm=\frac{1}{\|V\|_1^2}\big[g^\pm(\lambda)+c+\widetilde G_0(x,x_1)+\widetilde G_0(y,y_1) +\frac{\widetilde G_0(x,x_1) \widetilde G_0(y,y_1)}{g^\pm(\lambda)+c}\big] +E_2^\pm(\lambda),
$$
where
\begin{multline}\label{E2def}
E_2^\pm(\lambda):=\f1{\|V\|_1}\Big(1+\f{\widetilde G_0(x,x_1)}{g^\pm(\lambda)+c}\Big) E_0^\pm(\lambda)(y,y_1) +
\f1{\|V\|_1}\Big(1+\f{\widetilde G_0(y,y_1)}{g^\pm(\lambda)+c}\Big) E_0^\pm(\lambda)(x,x_1)\\+\f{E_0^\pm(\lambda)(x,x_1)E_0^\pm(\lambda)(y,y_1)}{g^\pm(\lambda)+c},
\end{multline}
and $\widetilde G_0=\|V\|_1G_0-c$. Using this and  \eqref{g form}, we have
$$
\mathcal R^--\mathcal R^+=-\frac{i}{2\|V\|_1}+c_3\frac{\widetilde G_0(x,x_1) \widetilde G_0(y,y_1)}{(\log(\lambda)+c_1)^2+c_2^2} +E_2^-(\lambda)-E^+_2(\lambda),
$$
where  $c_1,c_2,c_3\in\R$.

We consider first the contribution of the sine,
the contribution of the cosine term can be
controlled similarly but without the boundary term due
to the additional power of $\lambda$, 
see Proposition~\ref{freeevol}. We rewrite the $\lambda$ 
integral in \eqref{stone2_2} as a sum of 
\begin{align}\label{Sll1}
	&-\frac{i}{2\|V\|_1}  \int_0^\infty 
	\sin(t\lambda) \chi_1(\lambda) d\lambda,\\
	&\label{Sll2}
	\int_0^\infty \sin(t\lambda)  \chi_1(\lambda) \frac{\widetilde G_0(x,x_1) \widetilde G_0(y,y_1)}{(\log(\lambda)+c_1)^2+c_2^2} d\lambda,\\
	&\label{Sll3}
	\int_0^\infty \sin(t\lambda)  \chi_1(\lambda) [E_2^-(\lambda)-E^+_2(\lambda)] d\lambda.
\end{align}
Note that by \eqref{Aest} we have
\be\label{lem:Sll1}
\eqref{Sll1}=-\frac{i}{2t\|V\|_1}+O(t^{-2}).
\ee
The leading term above will cancel the boundary term that arose in Proposition~\ref{freeevol}.

The  decay rate $\f{1}{t\log^2(t)}$ appears in the
weighted $H^s$ setting in \cite{Kop} and appears in our
analysis due to the contribution of \eqref{Sll2},
as shown in the following lemma.

\begin{lemma}\label{lem:Sll2} For $t>2$, we have the bound
$$
|\eqref{Sll2}| \les \frac{1}{t\log^2(t)} k(x,x_1)  k(y,y_1)  (1+\log^+ |x|)(1+\log^+ |y|).
$$ 
\end{lemma}



\begin{proof}
First note that
\be\label{G0bound}
|\widetilde G_0(x,x_1)|\les 1+|\log|x-x_1||\les k(x,x_1) (1+\log^+ |x|).
\ee
Second, we bound the $\lambda$-integral by integrating
by parts.  Denote 
$\mathcal E(\lambda)=\frac{\chi_1(\lambda)}
{(\log(\lambda)+c_1)^2+c_2^2}$.
We note that $\mathcal E(0)=0$ and satisfies the 
following bounds
\begin{align*}
	|\partial_\lambda \mathcal E(\lambda)|  \les   
	\frac{\chi_1(\lambda)}{\lambda |\log(\lambda)|^3}, 
	\quad\quad \Big|\partial_\lambda^2 \mathcal E 
	(\lambda)\Big|  \les  \frac{\chi_1(\lambda)}
	{\lambda^2 |\log(\lambda)|^3}.
\end{align*}
Using these bounds, we will integrate by parts once,
then divide the integral into two pieces.  On the
first region we consider $0<\lambda <t^{-1}$ and perform
another integration by parts on the second region,
$\lambda>t^{-1}$.
\begin{multline*}
	\Big|\int_0^\infty \sin(t\lambda) \, \mathcal 
	E(\lambda) \, d\lambda \Big| \les 
	\f1t\int_0^{t^{-1}}|\mathcal E^\prime(\lambda)| 
	d\lambda+ 
	\Big|\frac{\mathcal{E}^\prime(t^{-1})}{t^{2}}\Big|
	+\f1{t^2}\int_{t^{-1}}^\infty \Big|\mathcal E^{\prime\prime}(\lambda)\Big| d\lambda\\
	\les \f{ 1}{t}\int_0^{t^{-1}} \frac{1}{\lambda |\log(\lambda)|^3} d\lambda +\f{1 }{t|\log(t)|^3}
	+\f{1}{t^2}\int_{t^{-1}}^\infty\frac{1}
	{\lambda^2 |\log(\lambda)|^3}  d\lambda.
\end{multline*}
A simple calculation shows that
$$
 \f1t \int_0^{t^{-1}} \frac{1}{\lambda |\log(\lambda)|^3} d\lambda \sim \f1{t|\log(t)|^2}.
$$
Finally we bound the integral on $[t^{-1},\infty)$:
\begin{multline*}
	\f{1}{t^2}\int_{t^{-1}}^\infty\frac{\chi_1(\lambda)}
	{\lambda^2 |\log(\lambda)|^3}  d\lambda \lesssim  
	\f1{t^2}+ \f1{t^2} \int_{t^{-1}}^{1/2}  
	\frac{1}{\lambda^2  |\log(\lambda)|^3 } d\lambda \\  
	\les \f1{t^2}+\f1{t^2} \int_{t^{-1/2}}^{1/2}  
	\frac{1}{\lambda^2  } d\lambda+\f1{t^2} 
	\int_{t^{-1}}^{t^{-1/2}}  \frac{1}{\lambda^2  
	|\log(t)|^3 } d\lambda
	\les \f1{t^{\f32}}+\f1{t|\log(t)|^3}.
\end{multline*}
The first inequality follows since the integral
converges on $[\f1{2},\infty)$.

Combining the bounds we obtained above finishes the proof of the lemma.
\end{proof}

\begin{lemma}\label{lem:Sll3} Let $0<\alpha<1/4$. For $t>2$, we have the bound
$$
|\eqref{Sll3}|\les t^{-1-\alpha} k(x,x_1)  k(y,y_1)  \big(\la x\ra \la y\ra \la x_1\ra \la y_1\ra\big)^{\f12+\alpha+}.
$$
\end{lemma}

\begin{proof}
We will only consider the following part of \eqref{Sll3}:
\begin{align}\label{gecic}
\int_0^\infty e^{it\lambda } \chi_1(\lambda) \Big(1+\f{\widetilde G_0(x,x_1)}{g(\lambda)+c}\Big) E_0(\lambda)(y,y_1) d\lambda=: \int_0^\infty e^{it\lambda}  \, \mathcal E(\lambda) \,d\lambda.
\end{align}
The other parts (and their derivatives)
are either of this   form or
much smaller. We also omit the $\pm$ signs since we can not rely on a cancellation between `+' and `-' terms.

Using Lemma~\ref{R0 exp cor}, Corollary~\ref{lipbound},  and \eqref{G0bound},  we estimate (for $0<\lambda<b\les \lambda<\lambda_1$)
$$
 |\partial_\lambda\mathcal E(\lambda)|\les k(x,x_1) (1+\log^+ |x|) \chi_1(\lambda) \lambda^{-\f12 } \la y-y_1\ra^{\f12 }\les  k(x,x_1) 
 (1+\log^+ |x|)\sqrt{ \la y\ra \la y_1\ra }
 \lambda^{-\f12 },
$$
and
\begin{multline*}
	\big| \partial_\lambda \mathcal E (b)-\partial_\lambda \mathcal E(\lambda) \big|\les \chi_1(\lambda) k(x,x_1) (1+\log^+ |x|) \lambda^{-\f12-\alpha} (b-\lambda)^\alpha \la y-y_1\ra^{\f12+\alpha } \\
	\les \chi_1(\lambda) k(x,x_1) (1+\log^+ |x|)
	(\la y\ra \la y_1\ra)^{\f12+\alpha }  \lambda^{-\f12-\alpha} (b-\lambda)^\alpha.
\end{multline*}

Noting that $\mathcal E(0)=0$ we can use 
Lemma~\ref{lem:ibp2} as in Proposition~\ref{freeevol} to 
obtain
$$
	|\eqref{gecic}|\les  \f1t\int_0^{\infty}\frac{|\mathcal E^\prime(\lambda)|}{  (1+\lambda t)} d\lambda
	+\f1{t}\int_{t^{-1/2}}^\infty \big|  \mathcal E^\prime(\lambda +\pi/t)-\mathcal E^\prime(\lambda)\big| d\lambda.
$$
Using the bounds above, we estimate the first integral by
$$
 \f{k(x,x_1) (1+\log^+ |x|) \sqrt{ \la y\ra \la y_1\ra } }{t}\int_0^{\infty}\frac{1}{ \sqrt{\lambda} (1+\lambda t)} d\lambda\les  \f{k(x,x_1)(1+\log^+ |x|)
 \sqrt{ \la y\ra \la y_1\ra } }{t^{1+\alpha}}.
$$
To estimate the second integral, we apply the Lipschitz bound with $b=\lambda+\pi/t$ to see
$$
	\f{ k(x,x_1) (1+\log^+ |x|)(\la y\ra \la y_1
	\ra)^{\f12+\alpha}}{t}\int_{t^{-1}}^{\lambda_1}  
	\lambda^{-\f12-\alpha} (t\lambda)^{-\alpha} d\lambda
	\les \f{ k(x,x_1) (1+\log^+ |x|)(\la y\ra \la 
	y_1\ra)^{\f12+\alpha }  }{t^{1+\alpha}},
$$
since $\alpha \in (0,\f14)$.

Taking into account the contribution of the term with the roles of $x$ and $y$ switched, and the bound
$1+\log^+|x|\les \la x\ra^{0+}$,
we obtain the assertion of the lemma.
\end{proof}


We are now prepared to prove the Proposition.

\begin{proof}[Proof of Proposition~\ref{prop:stone2_2}]

Using the bounds we obtained in \eqref{lem:Sll1}, Lemma~\ref{lem:Sll2},   Lemma~\ref{lem:Sll3} in \eqref{stone2_2}, we obtain
\begin{align*}
\eqref{stone2_2} = &-\frac{i}{2t\|V\|_1} \int_{\R^4} v(x_1)S(x_1,y_1)v(y_1) dx_1 dy_1\\
&+O\Big( \f{(1+\log^+ |x|)(1+\log^+ |y|) }{t\log^2(t)}   \int_{\R^4}  k(x,x_1) v(x_1)|S(x_1,y_1)| v(y_1)k(y,y_1)   dx_1 dy_1   \Big) \\
  &+ O\Big( \f{\big(\la x\ra \la y\ra\big)^{\f12+\alpha+}}{t^{1+\alpha} }    \int_{\R^4}
 k(x,x_1)  \la x_1\ra^{\f12+\alpha+}
 v(x_1)|S(x_1,y_1)| v(y_1) k(y,y_1)  \la y_1 \ra^{\f12+\alpha+}    dx_1 dy_1   \Big).
 \end{align*}
 Note that the integrals in the error terms are bounded in $x,y$, since 
$$ \sup_{y\in \R^2}
\|v(y_1)\la y_1\ra^{\f12+\alpha+} k(y,y_1)\|_{L^2_{y_1}}\les 1.$$
 Also note that we can replace $S$ with $P$ in the first integral since the other parts of the operator $S$ contains $Q$ on at least one side and  $Qv=vQ=0$. Therefore,
\begin{multline*}
\eqref{stone2_2}= -\frac{i}{2t\|V\|_1} \int_{\R^4} v(x_1)P(x_1,y_1)v(y_1) dx_1 dy_1\\+ O\Big( \f{(1+\log^+ |x|)(1+\log^+ |y|)  }{t\log^2(t)} \Big)+ O\Big( \f{\big(\la x\ra \la y\ra\big)^{\f12+\alpha+}}{t^{1+\alpha} } \Big)\\ =-\frac{i}{2t }+O\Big( \f{(1+\log^+ |x|)(1+\log^+ |y|)  }{t\log^2(t)} \Big)+O\Big( \f{\big(\la x\ra \la y\ra\big)^{\f12+\alpha+}}{t^{1+\alpha} } \Big).
\end{multline*}

\end{proof}

Next we consider the contribution of the third term in \eqref{resolvent id} to \eqref{stone2}:
\be\label{stone2_3}
 \int_{\R^4}\int_0^\infty \sin(t\lambda) \chi_1(\lambda)  [\mathcal R^-_2 -\mathcal R_2^+] v(x_1)[QD_0Q](x_1,y_1)v(y_1)
d\lambda dx_1 dy_1,
\ee
where
\be\label{calR2}
\mathcal R^\pm_2= R_0^{\pm}(\lambda^2)(x,x_1) R_0^{\pm}(\lambda^2)(y_1,y).
\ee
We notice that many of the terms that arise in this
term have zero contribution due to either the difference
of the `+' and `-' terms and the orthogonality of $Q$
and $v$.
Recall from  Lemma~\ref{R0 exp cor}  that
	\begin{align*}
		R_0^{\pm}(\lambda^2)(x,x_1)=c[a\log(\lambda|x-x_1|)+b\pm  i]+ E_0^\pm(\lambda)(x,x_1),
	\end{align*}
where $a,b,c\in\R$.  Therefore
\begin{multline*}
\mathcal R_2^\pm=c^2\big[(a\log(\lambda|x-x_1|)+b)(a\log(\lambda|y-y_1|)+b)-1 \big]\\  \pm i c^2 \big[ a\log(\lambda|x-x_1|)+a\log(\lambda|y-y_1|)+2b  \big] +E_3^\pm(\lambda),
\end{multline*}
where
\begin{multline}\label{E3def}
E_3^\pm(\lambda):= c[a\log(\lambda|x-x_1|)+b\pm  i] E_0^\pm(\lambda)(y,y_1)\\ +
c[a\log(\lambda|y-y_1|)+b\pm  i] E_0^\pm(\lambda)(x,x_1)+ E_0^\pm(\lambda)(x,x_1)E_0^\pm(\lambda)(y,y_1).
\end{multline}
Using this, we have
$$
\mathcal R_2^--\mathcal R_2^+=-2c^2(a\log(\lambda|x-x_1|)+a\log(\lambda|y-y_1|)+2b)+E_3^-(\lambda)-E^+_3(\lambda).
$$
Using this in \eqref{stone2_3}, and noting that the contribution of the first summand vanishes since $Qv=vQ=0$, we obtain
\be\label{stone2_3_2}
	\eqref{stone2_3} = \int_{\R^4}\int_0^\infty 
	\sin(t\lambda) \chi_1(\lambda)  
	[E_3^-(\lambda)-E^+_3(\lambda)] 
	v(x_1)[QD_0Q](x_1,y_1)v(y_1)\,
	d\lambda\, dx_1\, dy_1.
\ee

\begin{prop}\label{prop:stone2_3} Let $0<\alpha<1/4$. If $v(x)\les \la x\ra^{-\f32-\alpha-}$, then
we have
$$
|\eqref{stone2_3}|\les \f{\la x\ra^{\f12+\alpha+} 
\la y \ra^{\f12+\alpha+} }{t^{1+\alpha}}.
$$
\end{prop}
\begin{proof}
Let $\mathcal E(\lambda)= \chi_1(\lambda) E_3(\lambda)$ 
(we dropped the '$\pm$' signs).
Using
$$
|\log|x-x_1||\les k(x,x_1) (1+\log^+|x|),
$$
and the bounds in Lemma~\ref{R0 exp cor} and Corollary~\ref{lipbound} we estimate
(for $0<\lambda<b\les \lambda<\lambda_1$)
$$
|\partial_\lambda\mathcal E(\lambda)|
 \les \chi_1(\lambda) \lambda^{-\f12-} 
 (\la y \ra\la x\ra \la y_1\ra \la x_1\ra)^{\f12+}  
 k(x,x_1)k(y,y_1),
$$
$$
\big| \partial_\lambda \mathcal E (b)-\partial_\lambda \mathcal E(\lambda) \big|\les
 \chi_1(\lambda) k(x,x_1) k(y,y_1)   (\la x\ra\la x_1\ra\la y\ra \la y_1\ra)^{\f12+\alpha+ }  \lambda^{-\f12-\alpha-} (b-\lambda)^\alpha.
$$
Applying Lemma~\ref{lem:ibp2} together with these bounds as in the proof of the previous lemma, we bound the    $\lambda$-integral by
$$
  k(x,x_1) k(y,y_1)   (\la x\ra\la x_1\ra\la y\ra \la y_1\ra)^{\f12+\alpha+ }\f{1}{t^{1+\alpha}}.
$$
Therefore,
\begin{multline*}
\eqref{stone2_3}
\les t^{-1-\alpha} \int_{\R^4} k(x,x_1)k(y,y_1) (\la y \ra \la y_1\ra \la x\ra \la x_1\ra )^{\f12+\alpha+} v(x_1)|QD_0Q(x_1,y_1)|v(y_1) dx_1dy_1
\\ \les \f{\la x\ra^{\f12+\alpha+} \la y \ra^{\f12+\alpha+}}{t^{1+\alpha}},
\end{multline*}
since  $\|v(x_1)k(x,x_1) \la x_1\ra^{\f12+\alpha+}\|_{L^2_{x_1}}\les 1$.
\end{proof}

We now turn to the contribution of the error term $E^{\pm}(\lambda)$ from Lemma~\ref{Minverse}
in \eqref{resolvent id}. Dropping the `$\pm$' signs, we need to consider
\begin{align}\label{stone2_er}
	\int_{\R^4}  \int_0^\infty e^{it\lambda} 
	\, \mathcal E (\lambda)  v(x_1) v(y_1)  \, d\lambda
	\, dx_1\, dy_1,
\end{align}
where
$$
	\mathcal E (\lambda):= \chi_1(\lambda)
	R_0(\lambda^2)(x,x_1)E(\lambda)(x_1,y_1) 
	R_0(\lambda^2)(y,y_1).
$$
\begin{prop} \label{prop:stone2er} Let $0<\alpha<1/4$. If $v(x)\les \la x\ra^{-\f32-\alpha-}$, then
we have
  	$$
    	|\eqref{stone2_er}|\les \f{\la x\ra^{\f12+\alpha+} \la y \ra^{\f12+\alpha+}}{t^{1+\alpha}}.
  	$$
\end{prop}
\begin{proof}
Let
\begin{multline*}
		T_0:= \sup_{0<\lambda<\lambda_1} \lambda^{-\frac{1}{2} } |E^{\pm}(\lambda)|+\sup_{0<\lambda<\lambda_1} \lambda^{\frac{1}{2} } |\partial_\lambda E^{\pm}(\lambda)| \\ + \sup_{0<\lambda<b\les\lambda<\lambda_1}  \f{\lambda^{\frac{1}{2}+\alpha}}{ (b-\lambda)^{\alpha}} |\partial_\lambda E^{\pm}(b)-\partial_\lambda E^\pm(\lambda)|.
\end{multline*}

By Lemma~\ref{Minverse}, we see that $T_0 $ is  Hilbert-Schmidt on $L^2(\R^2)$, and hence we have the following bounds for the kernels
$$
|E^{\pm}(\lambda)|\les \lambda^{\f12} T_0,\,\,\,\,|\partial_\lambda E^{\pm}(\lambda)| \les \lambda^{-\f12} T_0,
$$
$$
|\partial_\lambda E^{\pm}(b)-\partial_\lambda E^\pm(\lambda)|\les \lambda^{-\frac{1}{2}-\alpha}(b-\lambda)^{\alpha} T_0,\,\,\,\,\,\,\text{ if } 0<\lambda<b\les \lambda<\lambda_1.
$$
Moreover, using Lemma~\ref{R0 exp cor}  and Corollary~\ref{lipbound}, we have (for $0<\lambda<b\les \lambda<\lambda_1$)
\begin{align*}
&|R_0(\lambda^2)(x,x_1)| \les (1+ |\log \lambda|) k(x,x_1) (1+\log^+|x|)\les \lambda^{0-}k(x,x_1) \la x\ra^{0+},\\
&|\partial_\lambda R_0(\lambda^2)(x,x_1)|\les \lambda^{-1} + \lambda^{-\f12} \sqrt{ \la x\ra \la x_1\ra}, \\
&|\partial_\lambda R_0(\lambda^2)(x,x_1)-\partial_\lambda R_0(b^2)(x,x_1)|\les (b-\lambda)^\alpha \Big[\lambda^{-(1+\alpha)} + \lambda^{-\f12} |x-x_1|^{\f12+\alpha}\Big].
\end{align*}
Therefore we have the bounds (for $0<\lambda<b\les \lambda<\lambda_1$)
\begin{align*}
&|\partial_\lambda\mathcal E(\lambda)| \les  \lambda^{-\f12-} (\la y \ra \la x\ra \la y_1\ra \la x_1\ra)^{\f12}  k(x,x_1)k(y,y_1)T_0(x_1,y_1),\\
&|\partial_\lambda \mathcal E (b)-\partial_\lambda \mathcal E (\lambda)|\les \lambda^{-\frac{1}{2}-\alpha-}(b-\lambda)^{\alpha} (\la y \ra \la x\ra \la y_1\ra \la x_1\ra)^{\f12+\alpha+}  k(x,x_1)k(y,y_1)T_0(x_1,y_1).
\end{align*}
Applying Lemma~\ref{lem:ibp2} as 
in the proof of Proposition~\ref{prop:stone2_3}
above yields the claim of the proposition.
\end{proof}

We close this section by noting that Proposition~\ref{freeevol}, Proposition~\ref{prop:stone2_2}, Proposition~\ref{prop:stone2_3}, and Proposition~\ref{prop:stone2er} yield  Proposition~\ref{prop:wtdhi} and hence Theorem~\ref{thm:wtd}.

\section{Zero not Regular}\label{sec:nonreg}

Finally we consider the evolution of \eqref{wave soln}
when zero is not a regular point of the spectrum of
$H=-\Delta+V$.  In particular, we prove 
Theorem~\ref{thm:nonreg}.  The discussion here is
not completely self-contained, the development of
the expansions as well as the spectral structure of
$-\Delta+V$ at zero energy in two spatial dimensions
are quite lengthy and technical.  We direct
the interested reader to \cite{JN,EG} for these 
details. 

If zero is not regular, 
we cannot use the resolvent expansions employed in
Sections~\ref{sec:low energy} or \ref{sec:weighted}.  However,
we can use the 
resolvent expansions
developed in \cite{EG} for the Schr\"odinger evolution
with obstructions at zero energy.  The proof of these
expansions  need not be adjusted to prove the dispersive
bounds for the wave equation, 
we cite them without
proof.  These
expansions have roots in the work of Jensen and
Nenciu, \cite{JN}.

Formally, see Theorem~6.2 of \cite{JN}, Section~5
of \cite{EG} or Defintion~\ref{resondef2} below, the
different types of resonances are defined in terms of
the non-invertibility of certain operators.  Alternatively, one can characterize the obstructions
in terms of certain spectral subspaces
of $L^2(\R^2)$.  The obstructions at zero energy can also
be related to distributional solutions to $H\psi=0$.
If $\psi \in L^\infty (\R^2)$ but $\psi \notin L^p(\R^2)$
for any $p<\infty$ we say there is an s-wave resonance
at zero.  If $\psi \in L^p(\R^2)$ for all $p\in(2,\infty]$
we say there is a p-wave resonance at zero.  Finally, if
$\psi\in L^2(\R^2)$ there is an eigenvalue at zero.  

We note that resonances at zero energy exist only for
Schr\"odinger operators in dimensions $n\leq 4$ and can
be characterized in terms of distributional solutions 
to $H\psi=0$ where the appropriate space for $\psi$
differs with the spatial dimension considered. 
For example, in dimensions three and four one has
$\langle \cdot \rangle^{-\beta}\psi \in L^2(\R^n)$
where $\beta =\frac{1}{2}+$ in dimension three or
$\beta=0+$ in dimension four.  In one spatial dimension,
a resonance is characterized by the Wronskian of the
Jost solutions, see \cite{GS}.

We note that it was shown in \cite{JN,EG} that a distributional solution of
$H\psi=0$ in two spatial dimensions
corresponding to a zero-energy resonance or
eigenvalue
must satsify $\psi\in L^\infty(\R^2)$ and
have the form
\begin{align*}
	\psi(x)=c_0+\frac{c_1 x_1+c_2x_2}{\la x\ra^2}+
	\Psi(x)
\end{align*}
Here $\Psi\in L^2$, and we write $x=(x_1,x_2)\in \R^2$.  The first term corresponds to a 
one-dimensional space of `s-wave' resonance functions and
the second term corresponds to a two-dimensional space
of `p-wave' resonances.  One can see that
if $c_0\neq 0$ there is an s-wave resonance at zero,
if $c_0=0$ but one of $c_1$ or $c_2$ is non-zero there
is a p-wave resonance, and finally if $c_0=c_1=c_2=0$
there is an eigenvalue at zero.

We note that in the case of the free equation, when 
$V\equiv 0$, there is an s-wave resonance at zero energy
due to the solution $\psi\equiv1$ to $-\Delta \psi=0$.
In spite of this low energy obstruction, the free
evolution still satisfies the $t^{-\f12}$ decay rate.
As in the case of the Schr\"odinger equation, we see
(in Theorem~\ref{thm:nonreg})
that the s-wave resonance does not affect the natural
rate of time
decay.

The key difference from the previous sections
when zero is not regular is that
the expansions for $M^{\pm}(\lambda)^{-1}$ given
in Lemmas~\ref{Minverse} and \ref{Minverse2} are no
longer valid.  The effect of such obstructions is a
strictly low energy phenomenon, and the bounds 
attained in the high energy regime, Section~\ref{sec:high energy}, hold with no adjustments. 
To prove our main result, we need only adjust the 
low energy argument to account for resonances and/or
eigenvalues at zero.  The development of such expansions
was first studied in the context of the Schr\"odinger
operator $e^{itH}$ in \cite{JN}, adapted to the study
of $L^1\to L^\infty$ dispersive estimate when zero is
regular in \cite{Sc2} and adapted to the case of
zero not regular in \cite{EG}.  We use these
expansions in \eqref{spectral rep}
to establish bounds for the wave equation.

The obstructive resonance and/or eigenvalues at zero  cause the spectral measure to
be more singular as $\lambda\to 0$.  Roughly
speaking, if there is an s-wave resonace, one has
a most singular term of size $\log\lambda$ as
$\lambda \to 0$, see Lemma~\ref{Mswave} below.  A p-wave resonance leads to 
singular terms of size $\lambda^{-2}(\log \lambda)^{-k}$,
for $k=1,2,\dots$, while an eigenvalue leads to singular
terms of the form $\lambda^{-2}$ plus singular terms
of the same form as when there is a p-wave resonance,
see Lemma~\ref{M_exp_cor2} below.
As in the previous sections, we use the Stone formula
and expansions for the resolvents to reduce the operator
bounds to showing certain oscillatory integral bounds.
In this section we work to bound the new
terms, which are
singular as $\lambda \to0$,  that arise
in the expansion(s) for $M^{\pm}(\lambda)^{-1}$ and note
that the remaining terms have clear analogues in the
expansions of Section~\ref{sec:low energy} and are bounded
as in the cases handled previously.

To develop a more precise expansion for
$M^{\pm}(\lambda)$.  We define the
operators $G_j$ for $j=1,2$ (see \eqref{Y0 def})
\begin{align}
	G_1f(x)&=\int_{\R^2} |x-y|^2f(y)\, dy,\label{G1 def}\\
	G_2f(x)&= \frac1{8\pi}\int_{\R^2} |x-y|^2\log|x-y|f(x)\, dy.\label{G2 def}
\end{align}

To make the above discussion more precise, 
we employ the terminology used
in \cite{EG} for the Schr\"odinger evolution.
Compare to Definition~\ref{resondef}.
\begin{defin}\label{resondef2}\begin{enumerate}
\item We say zero is a regular point of the spectrum
of $H = -\Delta+ V$ provided $ QTQ=Q(U + vG_0v)Q$ is invertible on $QL^2(\mathbb R^2)$.

\item Assume that zero is not a regular point of the spectrum. Let $S_1$ be the Riesz projection
onto the kernel of $QTQ$ as an operator on $QL^2(\R^2)$.
Then $QTQ+S_1$ is invertible on $QL^2(\mathbb R^2)$.  Accordingly, we define $D_0=(QTQ+S_1)^{-1}$ as an operator
on $QL^2(\R^2)$.
We say there is a resonance of the first kind at zero if the operator $S_1TPTS_1$ is invertible on
$S_1L^2(\mathbb R^2)$.

\item We say there is a resonance of the second kind at zero if $S_1TPTS_1$ is not invertible on
$S_1L^2(\R^2)$ but
$S_2vG_1vS_2$ is invertible
on $S_2L^2(\R^2)$, where $S_2$ is the Riesz projection onto the kernel of $S_1TPTS_1$.

\item Finally, if $S_2vG_1vS_2$ is not invertible on $S_2L^2(\R^2)$, we say there is a resonance of the third kind at zero.
We note that in this case the operator $S_3vG_2vS_3$ is always invertible on $S_3L^2$, where $S_3$ is the Riesz
projection onto the kernel of $S_2vG_1vS_2$ (see (6.41) in \cite{JN} or Section~5 of \cite{EG}).

\end{enumerate}

\end{defin}
We note that a resonance of the first kind corresponds
to an s-wave resonance only at zero.  A resonance of
the second kind corresponds to a p-wave resonance at 
zero but no eigenvalue at zero.  There may or may not be
an s-wave resonance at zero for a resonance of the second
kind.  Finally, a resonance of the third kind corresponds
to an eigenvalue at zero.  There may or may not be s-wave
or p-wave resonances at zero in this case.

Note that we used the operator $D_0$ in Section~\ref{sec:low energy} when zero was regular.  This
is not an abuse of notation since $S_1=0$ when there is
no resonance at zero.  We also note that $QD_0Q$ is
still absolutely bounded in these cases.  We also
define the operators
\begin{align*}
	D_1&:=(S_1TPTS_1+S_2)^{-1} \qquad \textrm{ on }
	S_1 L^2(\R^2),\\
	D_2&:=(S_2vG_1vS_2+S_3)^{-1} \qquad \textrm{ on }
	S_2 L^2 (\R ^2),\\
	D_3&:=(S_2vG_2vS_2)^{-1} \qquad\qquad\, \textrm{ on }
	S_3L^2 (\R ^2).\\
\end{align*}
The operators $D_j$ and the projections $S_j$ for
$j=1,2,3$ are all absolutely bounded operators.
\,
The invertibility of $S_1TPTS_1+S_2$ and
$S_2vG_1vS_2+S_3$ is clear from the definition of the
projections $S_2$ and $S_3$.  The invertibility of
$S_2vG_2vS_2$ was shown in Lemma~5.4 in \cite{EG}.
For a more complete discussion of these operators see
Section~5 of \cite{EG}.  In particular, these
operators are absolutely bounded on 
$S_j L^2(\R^2)$ for $j=1,2,3$ appropriately.
Finally, we note that
the projections $S_1-S_2$, $S_2-S_3$ and $S_3$ correspond
to the s-wave resonances, p-wave resonances and zero
eigenvalues respectively.

The characterization of resonances in 
Definition~\ref{resondef2} is useful when trying
to invert the operators $M^{\pm}(\lambda)$.  To invert
these operators one requires a technical inversion lemma,
see Lemma~2.1 in \cite{JN} and a longer expansion for
$M^{\pm}(\lambda)$ which we now state. 
By Lemma~2.2 of \cite{EG}
\begin{lemma} \label{lem:M_exp nr}
	For $\lambda>0$ define $M^\pm(\lambda):=U+vR_0^\pm(\lambda^2)v$.
	Let $P=v\langle \cdot, v\rangle \|V\|_1^{-1}$ denote the orthogonal projection onto $v$.  Then
	\begin{align*}
		M^{\pm}(\lambda)=g^{\pm}(\lambda)P+T+M_0^{\pm}(\lambda).
	\end{align*}
	Here $g^{\pm}(\lambda)=a\ln \lambda+z$ where $a\in\R\backslash\{0\}$ and $z\in\mathbb C\backslash\R$,
	and $T=U+vG_0v$ where $G_0$ is an
	integral operator defined in \eqref{G0 def}.
	Further, for any $\frac{1}{2}\leq k<2$,
	\begin{align*}
		M_0^{\pm}(\lambda)=\widetilde O_1(\lambda^k)
	\end{align*}
	if $v(x)\lesssim \langle x\rangle^{-\beta}$ for some $\beta>1+k$. Moreover,
	\begin{align}\label{M0_defn}
		M_0^{\pm}(\lambda)= g_1^{\pm}(\lambda)vG_1v
		+\lambda^2 vG_2 v+M_1^{\pm}(\lambda).
	\end{align}
	Here $G_1$, $G_2$ are integral operators defined in
	\eqref{G1 def}, \eqref{G2 def},  and
	$g_1^{\pm}(\lambda)=\lambda^2 (\alpha \log \lambda+\beta_{\pm})$ where $\alpha \in\R\backslash\{0\}$ and $\beta_\pm \in\mathbb C\backslash\R$.
  Further, for any $2<\ell<4$,
	\begin{align*}
		M_1^{\pm}(\lambda)=\widetilde O_1(\lambda^\ell)
	\end{align*}
	if $\beta>1+\ell$.
\end{lemma}
The difference between this expansion and the previous
expansion for $M^{\pm}(\lambda)$ in Lemma~\ref{lem:M_exp}
is the more detailed expansion of the `error term' 
$E_1^{\pm}(\lambda)$.

\subsection{The case of an s-wave resonance}
We first consider the case of a resonance of the
first kind (an s-wave resonance at zero only), that
is when there is a distributional solution to 
$H\psi=0$ with $\psi\in L^\infty$ and 
$\psi\notin L^p$ for any $p<\infty$.  As this is the type
of resonance that occurs for the free operator, we expect
to attain the $t^{-\f12}$ dispersive decay rate.

We note Corollary~2.7 in \cite{EG},
\begin{lemma}\label{Mswave}
 Assume that $|v(x)|\lesssim \langle x\rangle^{-1-k-}$ for some
$k\in [\frac{1}{2},2)$. Then in the case of a resonance of the first kind, we have
	\begin{multline*}
		M^{\pm}(\lambda)^{-1} =-h_{\pm}(\lambda)S_1D_1S_1-SS_1D_1S_1-S_1D_1S_1S\\
		 -h_{\pm}(\lambda)^{-1}SS_1D_1S_1S+
		h_{\pm}(\lambda)^{-1} S+QD_0Q +\widetilde O_1(\lambda^k),
	\end{multline*}
provided that $\lambda$ is sufficiently small.
\end{lemma}
We can now express the perturbed resolvent as in
\eqref{resolvent id} with the above expansion for
$M^{\pm}(\lambda)^{-1}$ in place of the expansion
used when zero is regular.  This new expansion can be
used in the Stone formula to deduce the desired bound
in Theorem~\ref{thm:nonreg}.

Roughly speaking, the analysis of the low energy
contribution when zero is regular controls all by the
the most singular terms
in the above expression.  That is, one needs to account
for the $-h_{\pm}(\lambda)$ terms.  Recall that
$h_{\pm}(\lambda)=a\log \lambda +z$ so a mild singularity
in the spectral measure is introduced as $\lambda \to 0$.
To control the remaining less singular terms
requires only very minor adjustments to
Lemma~\ref{r0lowbd}, Proposition~\ref{QDQ prop}, 
Lemma~\ref{lem:S} and Lemma~\ref{lem:Ebd}.  
Accordingly we first control the new singular terms
featuring $h_{\pm}(\lambda)$.

\begin{prop}\label{prop:summand2}
	If $|V(x)|\les \la x\ra^{-4-}$,
	we have the bound
\begin{multline*}
	\sup_{x,y\in\R^2}
  	\bigg|\int_{\R^4} \int_0^\infty (\sin(t\lambda) 
  	+\lambda\cos(t\lambda))
  	\chi_1(\lambda)\mathcal{K}(\lambda,|x-x_1|,|y-y_1|)\\
	v(x_1)S_1D_1S_1(x_1,y_1)
   	v(y_1) \,  d\lambda\, dx_1\, dy_1 \bigg|\lesssim t^{-\f12},
\end{multline*}
where
\begin{multline}  
	\mathcal{K}(\lambda,p,q) =h^+(\lambda) H^{+}_0(\lambda p)H_0^+(\lambda q)-h^-(\lambda)
	H^{-}_0(\lambda p)H_0^-(\lambda q) \\
	 =2ia\log(\lambda) [Y_0(\lambda p)J_0(\lambda q)+J_0(\lambda p)Y_0(\lambda q)]
+2z [J_0(\lambda p)J_0(\lambda q)+Y_0(\lambda p)Y_0(\lambda q)].\label{K defn}
\end{multline}
\end{prop}

\begin{proof}
This proof follows along the lines of the proofs
presented in Section~\ref{sec:low energy} with modifications to account for the singularities in the
spectral measure. 
The dispersive bound for the `low-low' interaction follows 
from the arguments of Proposition~3.2 in \cite{EG}
and the discussion in the proof of 
Proposition~\ref{QDQ prop}.  Where stationary phase
is used for Schr\"odinger, the wave equation admits
the following argument, in a `high-low' or
`high-high' interaction and occurs
when $t-s\leq \f t2$ where $s=p$, $q$ or $p+q$,
where $p=|x-x_1|$, $q=|y-y_1|$.
Consider a `high-low' interaction of the first term in
\eqref{K defn}, in the case that $t-p\leq \f t2$ 
we have to bound
\begin{multline*}
	\bigg|\int_0^\infty 
	(\sin(t\lambda)+\lambda \cos(t\lambda))
	\chi_1(\lambda)
	(\log\lambda) e^{-i\lambda p}\rho_+(\lambda p)
	F(\lambda, y_1,y)\, d\lambda\bigg|\\
	\les \int_0^{\lambda_1} 
	\log(\lambda)\frac{1}{(1+\lambda p)^{\f12}}\, 
	d\lambda
	\les \frac{1}{p^{\f12}}\int_0^{\lambda_1}
	\lambda^{-\f12-}\, d\lambda \les p^{-\f12}\les t^{-\f12}.
\end{multline*}
where  we used that $p\gtrsim t$ in the last step.
The `high-high' interaction term can be similarly bounded
when $t-(p+q)\geq \f t2$.
Without loss of generality, take $p\geq q$ then
\begin{multline*}
	\bigg|\int_0^\infty 
	(\sin(t\lambda)+\lambda \cos(t\lambda))
	\chi_1(\lambda)
	(\log\lambda) e^{-i\lambda (p+q)}\rho_+(\lambda p)
	\rho_+(\lambda q)\, d\lambda\bigg|\\
	\les \int_0^{\lambda_1} 
	\log(\lambda)\frac{1}{(1+\lambda p)^{\f12}}\, 
	d\lambda
	\les \frac{1}{p^{\f12}}\int_0^{\lambda_1}
	\lambda^{-\f12-}\, d\lambda \les p^{-\f12}\les t^{-\f12}.
\end{multline*}
Here we used that $|\rho_+(\lambda q)|\les 1$.

In the case when $t-p\geq \f t2$ we integrate by parts.
Due to the $\log \lambda$ term, we need to be slightly
more careful and take advantage of the fact that 
$S_1\leq Q$ we have
$S_1v=vS_1=0$.  We recall the following function and
estimates from Lemma~3.7 of \cite{EG}.
	Define
	\begin{align}\label{G2defn}
		\widetilde G^{\pm}(\lambda,x,x_1):=\widetilde\chi(\lambda |x-x_1|)\omega_{\pm}(\lambda |x-x_1|)
		-e^{\pm i \lambda (|x-x_1|-(1+|x|))}\widetilde\chi(\lambda (1+|x|))\omega_{\pm}(\lambda (1+|x|)).
	\end{align}
	with $\omega_{\pm}$ as in \eqref{largeJYH}.
	Then for any $0\leq \tau\leq 1$ and $\lambda\leq 2\lambda_1$,
	\begin{align*}
		&|\widetilde G^{\pm}(\lambda,x,x_1)|\les (\lambda 
		\la x_1 \ra)^{\tau} 
		\bigg(\frac{\widetilde\chi(\lambda |x-x_1|)}{(\lambda|x-x_1|)^{\frac{1}{2}}}
		+\frac{\widetilde\chi(\lambda (1+|x|))}{(\lambda (1+|x|))^{\frac{1}{2}}}\bigg),\\
		&|\partial_\lambda \widetilde G^{\pm}(\lambda,x,x_1)|\les \la x_1\ra \bigg(\frac{\widetilde\chi(\lambda |x-x_1|)}{( \lambda |x-x_1|)^{\frac{1}{2}}}
		+\frac{\widetilde\chi(\lambda (1+|x|))}{(\lambda 
		(1+|x|))^{\frac{1}{2}}}\bigg)
	\end{align*}
Taking $\tau>0$,
the smallness of $\widetilde G^{\pm}$ as $\lambda\to 0$
keeps the resulting function integrable if the derivative
acts on the $\log\lambda$.  According we bound with
\begin{align*}
	\frac{1}{t}\int_0^\infty \bigg|\frac{d}{d\lambda}
	\Big(\chi_1(\lambda) \log (\lambda)
	F(\lambda, x,x_1) \widetilde G^{\pm}(\lambda, y,
	y_1) \Big) \bigg|\, d\lambda
	&\les \frac{k(x,x_1) \la y_1 \ra}{t}
	 \int_0^{\lambda_1}
	1+|\log\lambda|+\lambda^{\tau-1} \, d\lambda\\
	&\les \frac{k(x,x_1) \la y_1 \ra}{t}
\end{align*}
Where we take $\tau=0+$ in the bounds on 
$\widetilde G$ to ensure integrability.  The full power of $\la y_1 \ra$ 
arises from when the derivative acts on 
$\widetilde G^{\pm}$, or if one simply takes 
$\tau=1$.
A similar bound for the `high-high' 
interaction when 
$t-(p+q)\geq \f t2$ can 
be found by replacing $F$ with
$\widetilde G$.  It is this term that necessitates
extra decay on the potential, as we now need
$\la y_1 \ra v(y_1)$ to be in $L^2(\R ^2)$, that is why 
we assume $v(y_1)\les \la y_1 \ra^{-2-}$.

\end{proof}

Finally with
$p=|x-x_1|$, $q=|x|+1$ we define
\begin{align*}
	G(\lambda,x,x_1)&:=\chi(\lambda p)
	J_0(\lambda p)-\chi(\lambda q) J_0(\lambda q).
\end{align*}
Then for any $\tau\in [0,1]$ and 
$\lambda \leq 2\lambda_1$ we have
\begin{align}\label{Gtrick}
	|G(\lambda,x,x_1)|&\lesssim  \lambda^{\tau} \langle x_1\rangle^{\tau},
	\qquad \,\,\,\,\,\,|\partial_\lambda G(\lambda,x,x_1)|\lesssim \langle x_1\rangle^\tau\lambda^{\tau-1}.
\end{align}
These bounds were proven in Lemma~3.3 of \cite{EG} and
allow us to extract further $\lambda$-smallness from the
Bessel function as $\lambda \to 0$.  This smallness is
needed for the $SS_1D_1S_1$ and $S_1D_1S_1S$ terms 
that appear in the expansion of $M^{\pm}(\lambda)^{-1}$
given in Lemma~\ref{Mswave}.  

\begin{lemma}\label{SS_1D_1S_1 prop}

	We have the bound
	\begin{multline}\label{QD0Q int2}
		\sup_{x,y\in \R^2}
		\bigg| \int_{\R^4}
		\int_{0}^\infty (\sin(t\lambda)+\lambda \cos(t\lambda)) \chi_1(\lambda)
		v(x_1)
		[SS_1D_1S_1(x_1,y_1)+S_1D_1S_1S(x_1,y_1)]
		v(y_1)\\
		\Big[
		R_0^+(\lambda^2)(x,x_1)R_0^+(\lambda^2)(y_1,y)
		-R_0^-(\lambda^2)(x,x_1)R_0^-(\lambda^2)(y_1,y)
		\Big] \, 
		d\lambda\, dx_1\, dy_1\bigg|
		\les t^{-\f12}.
	\end{multline}

\end{lemma}

\begin{proof}

This proof follows along the lines of the proof of
Proposition~\ref{QDQ prop}.  Due to the difference of
the `+' and `-' resolvents we need only bound integrals
of the form
\begin{align}\label{swave SS1}
	\int_{0}^\infty \sin(t\lambda) \chi_1(\lambda)
	Y_0(\lambda|x-x_1|)
	v(x_1) SS_1D_1S_1(x_1,y_1) v(y_1)	J_0(\lambda|y-y_1|)\, d\lambda,
\end{align}
along with a similar term with the Bessel functions
$J_0$ and $Y_0$ switching roles.  As usual, the sine
term is the most delicate to control.  The key difference
from Proposition~\ref{QDQ prop} is that we do not have
a projection orthogonal to the span of $v$ on both sides
of the operator $SS_1D_1S_1$, the worst case is when
$Y_0(\lambda |x-x_1|)$ is supported on 
$\lambda |x-x_1|\les 1$ in the above integral.  In this
case we cannot replace $Y_0$ with $F(\lambda, x,x_1)$
to control the logarithmic singularity as $\lambda \to 0$.
Instead we use \eqref{Gtrick} to gain integrability
at zero.  

As usual, we interpolate between two bounds to see the
$t^{-\f12}$ decay rate.  We bound $Y_0$ with
\eqref{log-bd} to see that
\begin{align*}
	|\eqref{swave SS1}|&\les k(x,x_1)\int_0^{\lambda_1}
	(1+|\log\lambda|) G(\lambda,y,y_1)\, d\lambda
	\les k(x,x_1).
\end{align*}
Here we took $\tau=0$ in \eqref{Gtrick} since 
$\log\lambda$ is integrable at zero.  On the other hand,
we can integrate by parts against 
$\sin(t\lambda)$ to gain time decay.
\begin{align*}
	|\eqref{swave SS1}|&\les \frac{1}{t}\int_0^{\lambda_1}
	\bigg| \frac{d}{d\lambda} \big(Y_0(\lambda|x-x_1|)
	G(\lambda,y,y_1)\big) \, d\lambda \bigg|\\
	&\les \frac{k(x,x_1)}{t} \int_0^{\lambda_1}
	\lambda^{-1}|G(\lambda,y,y_1)|+(1+|\log\lambda|)
	|\partial_\lambda G(\lambda,y,y_1)|\, d\lambda
	\les \frac{k(x,x_1)\la y_1 \ra^{\tau}}{t}.
\end{align*}
Selecting $\tau>0$ ensures that the boundary term 
vanishes as well as keeping the first term in the 
integrand integrable at $\lambda=0$.

The remaining terms to consider, `low-high' and `high-high' terms follow from the proof in Proposition~\ref{QDQ prop}, we leave the details to the
reader.

\end{proof}

We turn now to the terms in  Lemma~\ref{Mswave} with 
$h_{\pm}(\lambda)^{-1}$.  These terms are controlled
by either Lemma~\ref{lem:S} or the following bound
whose proof is identical.

\begin{lemma}\label{lem:S2}

	We have the bound
	\begin{multline}
		\bigg|\int_{\R^4}
		\int_0^\infty \sin(t\lambda)\chi_1(\lambda)
		\bigg( \frac{R_0^+(\lambda^2)(x,x_1) v(x_1)
		SS_1D_1S_1S(x_1,y_1)v(y_1)
		R_0^+(\lambda^2)(y_1,y)}{h_+(\lambda)}\\
		-\frac{R_0^-(\lambda^2)(x,x_1) v(x_1)
		SS_1D_1S_1S(x_1,y_1)v(y_1)
		R_0^-(\lambda^2)(y_1,y)}{h_-(\lambda)}\bigg)\, d\lambda\, dx_1\, dy_1
		\bigg| \les t^{-\f12}
	\end{multline}
	uniformly in $x$ and $y$.

\end{lemma}

The analysis for the remaining terms in Lemma~\ref{Mswave}
are controlled as
in  Section~\ref{sec:low energy}, see
Proposition~\ref{QDQ prop},
Lemmas~\ref{lem:S} and \ref{lem:Ebd}.  In 
particular one takes advantage of the identity
$Qv=vQ=0$ to employ the functions 
$F(\lambda, x,x_1)$, $\widetilde G(\lambda, x, x_1)$
and/or
$G(\lambda,x,x_1)$ 
in the place of
$Y_0$ and $J_0$ respectively.  Finally we note that the
statements of the propositions and lemmas hold if we
replace $\sin(t\lambda)$ with $\cos(t\lambda)\lambda$
as these terms and their derivatives are even smaller.
This suffices to prove the first statement in 
Theorem~\ref{thm:nonreg}.

\subsection{The case of a p-wave resonance 
and/or eigenvalue}

We now consider the case when there are non-trivial
solutions of $H\psi=0$ with either $\psi\in L^p$ for
all $p>2$, that is a p-wave resonance, or 
$\psi\in L^p$ for all $p\geq 2$ an eigenvalue at zero.
(Equivalently, when $S_1TPTS_1$ is not invertible on
$S_1 L^2$.)  These types of resonances correspond to a
loss of the time decay rate.

Let $\mathcal D=D_2+S_3D_3S_3vG_2vS_2D_2S_2vG_2vS_3D_3S_3-S_3D_3S_3vG_2vS_2D_2-D_2S_2vG_2vS_3D_3S_3$.
For the case of a p-wave resonance we note the expansion
in Corollary~4.2 of \cite{EG},
\begin{lemma}\label{M_exp_cor2}

	Assume that $v(x)\les \langle x\rangle^{-3-}$.  Then,
	in the case of a resonance of the second kind, we have
	\begin{multline}\label{M pwave}
		M^{\pm}(\lambda)^{-1} =\frac{S_2D_2S_2}{g_1^{\pm}(\lambda)}
		+Q\Gamma_{1}^{\pm}(\lambda)Q+Q\Gamma_2^{\pm}(\lambda) + \Gamma_3^{\pm}(\lambda)Q+
		 \Gamma_4^{\pm}(\lambda) \\
		  +(M^{\pm}(\lambda)+S_1)^{-1}+\widetilde O_1(\lambda^{2-}),
	\end{multline}
	where
	$\Gamma_{i}^{\pm}$, $i=1,2,3,4$ are absolutely bounded operators on $L^2(\R^2)$ with
	$\Gamma_{1}^{\pm}(\lambda)=O(\lambda^{-2}( \log \lambda)^{-2})$,
	$\Gamma_{2}^{\pm}(\lambda), \Gamma_3^{\pm}(\lambda)=O(\lambda^{-2}( \log \lambda)^{-3})$, and
	$\Gamma_{4}^{\pm}(\lambda)=O(\lambda^{-2}( \log \lambda)^{-4})$.

 In the case of a resonance of the third kind, we have
	\begin{multline}\label{M eval}
		M^{\pm}(\lambda)^{-1} =\frac{S_3D_3S_3}{\lambda^2 }+\frac{S_2\mathcal D S_2}{ g_1^\pm(\lambda)}
		+Q\Gamma_{1}^{\pm}(\lambda)Q+Q\Gamma_2^{\pm}(\lambda) + \Gamma_3^{\pm}(\lambda)Q+
		 \Gamma_4^{\pm}(\lambda) \\
		+(M^{\pm}(\lambda)+S_1)^{-1}+\widetilde O_1(\lambda^{2-}),
	\end{multline}
	where $\mathcal D$ is as above, and $\Gamma_i$ are absolutely bounded operators on $L^2(\R^2)$.  These operators are distinct from the
	$\Gamma_i $ in the case of a resonance of the second kind, but satisfy the same size estimates.

\end{lemma}
We note that the operators $(M^{\pm}(\lambda)+S_1)^{-1}$
have an expansion nearly identical in form to that we
use in Lemma~\ref{Minverse}, see Lemma~2.5 in \cite{EG}.
The error term is slightly different, but obeys the
same error bounds.
The resulting terms behave as in those already bounded in 
Section~\ref{sec:low energy}.

The new singular terms in \eqref{M pwave} and \eqref{M eval}
are all surronded by projections
which are orthogonal to  the span of $v$, thus we can
use the
functions $F(\lambda, x, x_1)$  
$G(\lambda,x,x_1)$, 
and $\widetilde G(\lambda,x,x_1)$ defined in the previous 
section to gain extra $\lambda$ smallness near 
$\lambda=0$.

The solution operator can be shown be bounded if there are
p-wave resonances and/or eigenvalues at zero energy.
The proof is
nearly identical to that found in Section~4 of
\cite{EG} for the Schr\"odinger equation.  
This follows as the oscillatory nature
of the imaginary Gaussian $e^{it\lambda^2}$ in the
integral is not used in \cite{EG} for
these terms, and no integration
by parts is used due to the highly singular nature of
integrals as $\lambda \to 0$.  In our case we do not
use any oscillation of $\cos(t\lambda)$.
We provide the details for
the convience of the reader.  We first establish the
bounds for the cosine operator, then describe how one
can reduce the contribution of the sine operator to
the bounds for the cosine operator.

Accordingly, it suffices to establish the estimates for the contributions of the terms:
\begin{align*}
  \frac{S_2D_2S_2}{g_1^{\pm}(\lambda)}
		+Q\Gamma_{1}(\lambda)Q+Q\Gamma_2(\lambda) + \Gamma_3(\lambda)Q+
		\Gamma_4(\lambda).
\end{align*}
For the rest of the analysis, the standing assumption is
that the potential satisfies the decay $|V(x)|\les 
\la x\ra^{-6-}$.
We start with the following.

\begin{lemma}\label{D2 lemma}

	We have the bound
  	\begin{multline}\label{D2 int bound}
    		\sup_{x,y\in \R^2}
    		\bigg| \int_{\R^4}\int_0^\infty \cos(t\lambda) \lambda \chi_1(\lambda)\bigg[
		\frac{R_0^{+}(\lambda^2)(x,x_1)v(x_1)
		S_2D_2S_2(x_1,y_1) v(y_1) R_0^{+}(\lambda^2)}{g_1^{+}(\lambda)}\\
    		-\frac{R_0^{-}(\lambda^2)(x,x_1) v(x_1) S_2D_2S_2(x_1,y_1)  v (y_1)R_0^{-}(\lambda^2)(y_1,y)}
    		{g_1^{-}(\lambda)}\bigg]
		\, d\lambda\, dx_1\, dy_1 \bigg|
		\les 1.
  	\end{multline}

\end{lemma}

\begin{proof}

  We note that we must exploit some cancellation between the `$+$' and `$-$' terms.  Recall that
  $H_0^{\pm}(y)=J_0(y)\pm iY_0(y)$ and the definition of $g_1^{\pm}(\lambda)$ in Lemma~\ref{lem:M_exp} give us
  \begin{align}
  \label{pwave hankel cancel}  \frac{R_0^{+}(\lambda^2) R_0^{+}(\lambda^2)}{g_1^{+}(\lambda)}
    -\frac{R_0^{-}(\lambda^2)  R_0^{-}(\lambda^2)}{g_1^{-}(\lambda)}
    &=\frac{J_0(\lambda p)J_0(\lambda q)-Y_0(\lambda p)Y_0(\lambda q)}{\lambda^2[(\log \lambda +c_1)^2+c_2^2]} \\
    &+\frac{(J_0(\lambda p)Y_0(\lambda q)+Y_0(\lambda p)J_0(\lambda q)) (\log \lambda +c_1)}
    {\lambda^2[(\log \lambda +c_1)^2+c_2^2]}\nn
  \end{align}
    We again must consider cases based on 
  the supports of the resolvents.  That is, we need to
  consider `low-low', `high-low' and `high-high' 
  interactions.
  Let us first consider the case when both resolvents are supported on the low part of their arguments.
  Contribution of the first term in \eqref{pwave hankel cancel} satisfies the required bound since $|J_0(z)|\les 1$, and $\frac{1}{\lambda(\log\lambda)^2}$ is integrable on $[0,\lambda_1]$. Since the other terms have additional powers $\log\lambda$ in the numerator, we need to use 
  the relation $S_2v=vS_2=0$ (recall that $S_2 \leq Q$).

  Consider the contribution of the second term  in \eqref{pwave hankel cancel}. We replace $\chi Y_0$ with $F(\lambda,\cdot,\cdot)$, and using the bounds on $F$, we see
 $$
        \bigg|\int_0^\infty \chi_1(\lambda) \frac{F(\lambda, x, x_1) F(\lambda,y,y_1)}
      {\lambda [(\log \lambda +c_1)^2+c_2^2]}\, d\lambda \bigg|
      \lesssim k(x,x_1)k(y,y_1) \int_0^{\lambda_1} \frac{1}
      {\lambda (\log \lambda)^2}\, d\lambda
      \lesssim k(x,x_1)k(y,y_1).
$$
  The mixed $J_0$ and $Y_0$ terms in the second part
  of \eqref{pwave hankel cancel} are bounded similarly using $|G(\lambda,x,x_1)|\les \lambda^{0+}\langle x_1\rangle^{0+}$ to ensure integrability.

When one or both of the Bessel functions is supported on high energies, we use the functions $\widetilde G(\lambda,p,q)$.
The bound $|\widetilde G(\lambda,x,x_1)|\les \lambda^{0+}\la x_1 \ra^{0+}$ suffices for obtaining the required bound.
  The details are left to the reader.
\end{proof}

\begin{lemma}\label{low high YJ lemma3}

	For $\mathcal C_i(z)=J_0(z)$ or $Y_0(z)$ for $i=1,2$, we have the bound
	\begin{multline*}
		\sup_{x,y\in \R^2}
    	\bigg| \int_{\R^4} \int_0^\infty 
    	\cos(t\lambda) \lambda \chi_1(\lambda)
		\mathcal C_1(\lambda|x-x_1|)v(x_1)\\ Q \Gamma_1(\lambda) Q (x_1,y_1) v(y_1)
		\mathcal C_2(\lambda|y-y_1|) \, d\lambda \, dx_1\, dy_1
		\bigg|\les 1.
  	\end{multline*}

\end{lemma}

\begin{proof}

	Unlike in Lemma~\ref{D2 lemma}
	we do not need to use any cancellation between the `$+$' and `$-$' terms.  We consider the terms that arise when both
	$\mathcal C_1$ and $\mathcal C_2$ are supported on small energies.  Consider,
	\begin{align*}
		\int_0^\infty \cos(t\lambda)\lambda \chi_1(\lambda) \chi(\lambda p)\mathcal C_1(\lambda p)
		vQ\Gamma_1(\lambda)Qv \chi(\lambda q)\mathcal C_2(\lambda q)
		\, d\lambda,
	\end{align*}
	where $p=|x-x_1|$, $q=|y-y_1|$. In the worst case when $\mathcal C_1=\mathcal C_2=Y_0$, 
	as $Qv=vQ=0$, we replace  $\chi Y_0 $ with $F $   to obtain
\begin{multline}
\label{FF Ci low}
		\bigg|\int_0^{\lambda_1} \lambda  F(\lambda,x,x_1) \Gamma_1(\lambda)F(\lambda,y,y_1)\, 
		d\lambda\bigg|\\
		 \lesssim \bigg| \int_0^{\lambda_1}
		 \frac{F(\lambda,x,x_1)F(\lambda,y,y_1)}{\lambda (\log\lambda)^2}\, d\lambda\bigg|
		 \sup_{0<\lambda<\lambda_1}|\lambda^2 (\log \lambda)^2 \Gamma_1(\lambda)|
		 \\ \lesssim k(x,x_1) k(y,y_1) \sup_{0<\lambda<\lambda_1}|\lambda^2 (\log \lambda)^2 \Gamma_1(\lambda)|.
	\end{multline}
	The last line follows since
  $\sup_{0<\lambda<\lambda_1} |\lambda^2 (\log \lambda)^2\Gamma_1(\lambda)|$
	defines a bounded operator on $L^2(\R^2)$ (by Lemma~\ref{M_exp_cor2}), we are done.
	The other low energy terms are similar using $G$ instead of $F$.

	For the large energies, we note that the argument runs in a similar manner.  Using
	$\widetilde\chi(z) (|J_0(z)|+|Y_0(z)|)\lesssim 1$, and an argument as in \eqref{FF Ci low}, it easily follows that  the integral is bounded
	as desired.
\end{proof}

We need the following modified version  
of Lemma~\ref{low high YJ lemma3} to control
the other $\Gamma_{i} (\lambda)$
terms, which don't have the projection $Q$ on both sides.
It is worth noting that the loss of a $Q$ on either side
occurs exactly with the gain of a $(\log \lambda)^{-1}$
smallness as $\lambda \to 0$, 
which is exactly the gain one gets from
replacing $Y_0$ with $F$.

\begin{corollary}\label{low high YJ cor3}

	For $\mathcal C_i(z)=J_0(z)$ or $Y_0(z)$ for $i=1,2$, we have the bound
	$$
		\sup_{x,y\in \R^2}
    		\bigg|  \int_{\R^4} \int_0^\infty \cos(t\lambda) \lambda \chi_1(\lambda)
		\mathcal C_1(\lambda|x-x_1|)v(x_1) Q\Gamma_2(\lambda)(x_1,y_1)   v(y_1)
\mathcal C_2(\lambda|y-y_1|) \, d\lambda  \, dx_1\, dy_1 \bigg|\les 1.$$

  The same bounds hold when $Q\Gamma_2(\lambda) $ is replaced by
  	$ \Gamma_3(\lambda)Q$ or $ \Gamma_4(\lambda) $.

\end{corollary}

\begin{proof}
	We repeat the analysis of Lemma~\ref{low high YJ lemma3}.
 Consider the case when both $\mathcal C_i(\lambda \cdot)$  are supported on
	low energies and both are $Y_0$.    We note that when $\lambda<1$, we have
	\begin{align}\label{Y0 low bd}
		|Y_0(\lambda p)\chi(\lambda p)| \lesssim (1+|\log\lambda|)(1+\log^- p).
	\end{align}
	Using this and replacing $\chi Y_0$ with $F$ on one side, we obtain the bound
	\begin{multline*}
      	  \int_0^{\lambda_1} \frac{|F(\lambda, x,x_1)| (1+\log^{-} q)}{\lambda |\log \lambda|^2}\, d\lambda
		    \sup_{0<\lambda<\lambda_1} |\lambda^2 (\log \lambda)^3 \Gamma_2(\lambda)|\\
      		 \lesssim  k(x,x_1) k(y,y_1)
		 \sup_{0<\lambda<\lambda_1} |\lambda^2 (\log \lambda)^3 \Gamma_2(\lambda)|.
  	\end{multline*}
	The same bound holds for  $\Gamma_3(\lambda)Q$.  For the contribution of $ \Gamma_4(\lambda) $, we have
	\begin{align*}
      		&\bigg|\int_0^\infty \lambda \chi(\lambda) Y_0(\lambda p)  \Gamma_4(\lambda)  Y_0(\lambda q)
      		\, d\lambda \bigg| \\
      		&\lesssim  \int_0^{\lambda_1} \frac{(1+|\log \lambda|)(1+\log^{-} p)(1+|\log \lambda|)(1+\log^{-} q)}
		{\lambda |\log \lambda|^4}\, d\lambda   \sup_{0<\lambda<\lambda_1} |\lambda^2 (\log \lambda)^4\Gamma_4(\lambda)|\\
      		&\lesssim k(x,x_1) k(y,y_1)\sup_{0<\lambda<\lambda_1} |\lambda^2 (\log \lambda)^4\Gamma_4(\lambda)|.
  	\end{align*}
  The other cases are similar.

When one of the $\mathcal C_i(\lambda \cdot)$ is supported
	on high energies, the analysis is less delicate. The required bound follows from $\widetilde\chi(z)(|J_0(z)|+|Y_0(z)|)\les 1$.
\end{proof}
This completes the proof in the case of a resonance of the second kind.

We note that the above bounds in Lemma~\ref{low high YJ lemma3} and Corollary~\ref{low high YJ cor3}
also hold for the $\Gamma_i$ term in \eqref{M eval}.  Thus for a resonance of the third kind,
it suffices to consider the leading $\lambda^{-2}$ term in \eqref{M eval}.  Noting \eqref{K defn} and
the fact that the kernel of $D_3$ is real-valued, the following lemma completes the analysis. 

\begin{lemma}\label{D3 lemma}

	We have the bound
	\begin{align}\label{low high bound6}
    	\sup_{x,y\in \R^2}
    	\bigg| &\int_{\R^4} \int_0^\infty \cos(t\lambda) \lambda \chi_1(\lambda)
		J_0(\lambda|x-x_1|)v(x_1) \frac{S_3D_3S_3 }{\lambda^2}  v(y_1)  Y_0(\lambda|y-y_1|) \, d\lambda   \, dx_1\, dy_1 \bigg|\les 1.
  	\end{align}
\end{lemma}

\begin{proof}

	We provide a sketch of the proof.  Due to similarities to previous proofs, we leave some details to the reader.
	We again consider the case when the Bessel functions are supported on low energy first.  Accordingly,
	we wish to control
	\begin{align*}
		\bigg|\int_0^\infty \cos(t\lambda)\lambda &\chi_1(\lambda) \chi(\lambda p) J_0(\lambda p)v(x_1)\frac{D_3}{\lambda^2}
		v(y_1) \chi(\lambda q) Y_0(\lambda q) \, d\lambda\bigg|\\
		&\lesssim \bigg| \int_0^\infty \chi_1(\lambda) \frac{G(\lambda,x,x_1)F(\lambda,y,y_1)}{\lambda}\, d\lambda
		\lesssim \langle x_1\rangle^{\tau}k(y,y_1).
	\end{align*}
	Where we used $S_3\leq Q$ and the known bounds
	for $G$ with $\tau>0$.
	
	For the case when one function is supported on high energy, we have
	\begin{align*}
		\bigg|\int_0^\infty \cos(t\lambda)\lambda &\chi_1(\lambda) \widetilde\chi(\lambda p)
		J_0(\lambda p)v(x_1)\frac{D_3}{\lambda^2}
		v(y_1) \chi(\lambda q) Y_0(\lambda q) \, d\lambda\bigg|\\
		&\lesssim \bigg| \int_0^\infty \chi_1(\lambda) \frac{\widetilde G(\lambda,x,x_1)F(\lambda,y,y_1)}{\lambda}\, d\lambda
		\lesssim \langle x_1\rangle^{\tau}k(y,y_1).
	\end{align*}
	Similarly one uses $\widetilde G(\lambda,y,y_1)$ instead of $F(\lambda, y,y_1)$ if we have $\widetilde\chi(\lambda q)$.
	
	When both functions are supported on high energy, we have
	\begin{align*}
		\bigg|\int_0^\infty \cos(t\lambda)\lambda &\chi_1(\lambda) \widetilde\chi(\lambda p)
		J_0(\lambda p)v(x_1)\frac{D_3}{\lambda^2}
		v(y_1) \widetilde\chi(\lambda q) Y_0(\lambda q) \, d\lambda\bigg|\\
		&\lesssim \bigg| \int_0^\infty \chi(\lambda) \frac{\widetilde G(\lambda,x,x_1)\widetilde G(\lambda,y,y_1)}{\lambda}\, d\lambda
		\lesssim \langle x_1\rangle^{\tau}\langle y_1\rangle^{\tau}.
	\end{align*}

\end{proof}

We are now ready to prove the main theorem of this section.

\begin{theorem}\label{pwave thm}

	Let $V:\R^2\to \R$ be such that $|V(x)|\lesssim \langle x\rangle^{-\beta}$ for some $\beta>6$.  Further assume that
	$H=-\Delta+V$ has a resonance of the second or third kind at zero energy.  Then, there is a time dependent operator
	$F_t$ such that
	\begin{align*}
		\sup_t \| F_t\|_{L^1\to L^\infty}\lesssim 1,
		\qquad \|\cos(t\sqrt{H})P_{ac} - F_t \|_{L^1\to L^\infty}\lesssim |t|^{-1},\,\,\,\,|t|>1.
	\end{align*}

\end{theorem}

\begin{proof}
If we denote the terms that arise from the contribution of the terms in the first lines of  \eqref{M pwave} and \eqref{M eval} as $F_t$,
Lemmas~\ref{D2 lemma}, \ref{low high YJ lemma3}, and~\ref{D3 lemma}
and Corollary~\ref{low high YJ cor3} show that
\begin{align*}
	\sup_t \| F_t\|_{L^1\to L^\infty}\lesssim 1.
\end{align*}

As the remaining terms in \eqref{M pwave} and \eqref{M eval} are identical in form to those that arise in the analysis of a resonance of the first kind, we can use the bounds from the previous subsection to   establish the theorem.
\end{proof}

The sine operator can be seen to obey similar bounds,
but with a bound that, at worst, grows linearly in time.
To see this, we need only bound the most singular
terms of the sine evolution.  We use the following
inequality to reduce the analysis to terms of the form
previously considered
\begin{align*}
	\bigg|\int_0^{\infty} \sin(t\lambda)\frac{\mathcal E(\lambda)}{\lambda}\, d\lambda\bigg|
	&\les |t|\int_0^\infty |\mathcal E(\lambda)|\, d\lambda.
\end{align*}
Where we used the crude bound $|\sin(t\lambda)|\leq |t\lambda|$.  

\appendix
\setcounter{section}{0}



\vspace{1cm}
\begin{large}
\noindent
{\bf Acknowledgment. \\}
\end{large}
The author would like to thank Wilhelm Schlag for 
suggesting this line of investigation, and the anonymous
referee whose careful reading and suggestions 
greatly improved the paper.

\end{document}